\documentclass[12pt]{article}
\usepackage{fullpage}
\usepackage{amsmath,amssymb,amsfonts,amsthm}
\usepackage[all,arc,knot,poly]{xy}
\usepackage{enumerate}
\usepackage{mathrsfs}
\usepackage{mathtools}
\usepackage{amsxtra}
\usepackage{needspace}
\usepackage{graphicx}
\usepackage{setspace}

\DeclareMathOperator{\Ad}{\operatorname{Ad}}
\DeclareMathOperator{\Hom}{\operatorname{Hom}}
\DeclareMathOperator{\Frac}{\operatorname{Frac}}
\DeclareMathOperator{\sgn}{\operatorname{sgn}}
\DeclareMathOperator{\supp}{\operatorname{supp}}

\newcommand\sbullet[1][.5]{\mathbin{\vcenter{\hbox{\scalebox{#1}{$\bullet$}}}}}

\theoremstyle{definition}
\newtheorem{theorem}{Theorem}[section]
\newtheorem{definition}[theorem]{Definition}
\newtheorem{lemma}[theorem]{Lemma}
\newtheorem{proposition}[theorem]{Proposition}

\newtheorem{corollary}[theorem]{Corollary}

\title{Quantization of canonical bases and the \protect\\quantum symplectic double}

\author{
  Dylan G.L. Allegretti
}

\date{}

\begin{document}

\maketitle

\begin{abstract}
We describe a natural $q$-deformation of Fock and Goncharov's canonical basis for the algebra of regular functions on a cluster variety associated to a quiver of type~$A$. We then describe an extension of this construction involving a cluster variety called the symplectic double.
\end{abstract}

\tableofcontents

\section{Introduction}

\subsection{The cluster Poisson variety}

Cluster varieties are geometric objects introduced by Fock and Goncharov in a series of papers~\cite{IHES,dual,ensembles,dilog} borrowing combinatorial ideas from Fomin and Zelevinsky's theory of cluster algebras~\cite{FZI}. They arise naturally as moduli spaces of local systems on surfaces. More~precisely, one considers a \emph{marked bordered surface}, defined as a compact oriented surface~$\mathbb{S}$ with boundary, together with a finite set $\mathbb{M}\subset\mathbb{S}$ of marked points, such that any boundary component of~$\mathbb{S}$ contains at least one marked point. Given a marked bordered surface $\Sigma=(\mathbb{S},\mathbb{M})$, Fock and Goncharov defined a moduli space parametrizing $PGL_2(\mathbb{C})$-local systems on~$\mathbb{S}\setminus\mathbb{M}$ with additional framing data associated to the marked points. An \emph{ideal triangulation} of $\Sigma$ is defined as a triangulation of~$\mathbb{S}$ whose vertices are precisely the points of~$\mathbb{M}$. For any marked bordered surface $\Sigma$ admitting an ideal triangulation, the moduli space of framed local systems on~$\Sigma$ is birational to a cluster variety $\mathcal{X}_\Sigma$ known as the \emph{cluster Poisson variety}.

One of the main ideas from the work of Fock and Goncharov is that in many cases the algebra of regular functions on the cluster Poisson variety possesses a canonical vector space basis. To construct this canonical basis, Fock and Goncharov considered a dual cluster variety denoted $\mathcal{A}_\Sigma$ and known as the \emph{cluster $K_2$-variety}~\cite{double}. This cluster variety is birational to a certain moduli space of $SL_2(\mathbb{C})$-local systems associated to~$\Sigma$. One can form the set $\mathcal{A}_\Sigma(\mathbb{Z}^t)$ of points of this cluster variety valued in the semifield $\mathbb{Z}^t$ of \emph{tropical integers}, and there is a canonical map 
\[
\mathbb{I}_{\mathcal{A}}:\mathcal{A}_\Sigma(\mathbb{Z}^t)\rightarrow\mathcal{O}(\mathcal{X}_\Sigma)
\]
from this set into the algebra of regular functions on~$\mathcal{X}_\Sigma$~\cite{IHES,dual}. Building on the ideas of~\cite{IHES}, it has been shown in many cases that the image of this map is a canonical basis for the algebra of regular functions (see~\cite{categorification}, Theorem~1.1).

This canonical basis construction is closely related to Lusztig's canonical bases in the theory quantum groups~\cite{Lusztig}, as well as the work of Musiker, Schiffler, and Williams~\cite{MSW2} and Gross, Hacking, Keel, and Kontsevich~\cite{GHKK} on canonical bases for cluster algebras. It has also found applications in mathematical physics where it was used by Gaiotto, Moore, and Neitzke to count framed BPS~states in certain four-dimensional supersymmetric quantum field theories~\cite{GMN}.

For any ideal triangulation $T$ of~$\Sigma$, there is a collection of rational coordinates $X_j:\mathcal{X}_\Sigma\dashrightarrow\mathbb{C}^*$ on the cluster Poisson variety indexed by the set $J_T$ of internal edges of~$T$. These coordinates are known as \emph{cluster coordinates} or \emph{Fock-Goncharov coordinates}. If we number the edges of $T$ so that $J_T=\{1,\dots,n\}$, then the canonical function $\mathbb{I}_{\mathcal{A}}(l)$ associated to a point $l\in\mathcal{A}_\Sigma(\mathbb{Z}^t)$ can be written as a Laurent polynomial 
\[
\mathbb{I}_{\mathcal{A}}(l)=\sum_{\mathbf{d}=(d_1,\dots,d_n)\in\mathbb{Z}^n}\lambda_{\mathbf{d}}\cdot X_1^{d_1}\dots X_n^{d_n}
\]
where the coefficients $\lambda_{\mathbf{d}}$ are nonnegative integers and only finitely many terms are nonzero. This expansion includes a term of the form $X_1^{a_1}\dots X_n^{a_n}$ which is the highest term in the sense that, for any other term $\lambda_{\mathbf{d}}\cdot X_1^{d_1}\dots X_n^{d_n}$ in the expansion, we have $d_j\leq a_j$ for all $j\in J_T$.

An important feature of the cluster Poisson variety is that this object carries a canonical Poisson structure. For any ideal triangulation~$T$ of~$\Sigma$, there is a skew symmetric matrix~$\varepsilon_{ij}$~($i,j\in J_T$) such that the Poisson bracket is given on coordinates by 
\[
\{X_i,X_j\}=\varepsilon_{ij}X_iX_j.
\]
In particular, the cluster Poisson variety can be canonically quantized~\cite{ensembles, dilog}, giving rise to a family of noncommutative algebras $\mathcal{X}_\Sigma^q$ depending on a parameter $q$ such that $\mathcal{X}_\Sigma^q$ coincides with the function field $\mathbb{Q}(\mathcal{X}_\Sigma)$ when $q=1$. For any ideal triangulation $T$ of~$\Sigma$, the algebra $\mathcal{X}_\Sigma^q$ is isomorphic to the noncommutative fraction field of a quantum torus, generated over~$\mathbb{Q}[q,q^{-1}]$ by noncommuting variables $X_j^{\pm1}$, $j\in J_T$, subject to the relations 
\[
X_iX_j=q^{2\varepsilon_{ij}}X_jX_i
\]
for all $i$,~$j\in J_T$.

In this paper, we will study the quantization of~$\mathcal{X}_\Sigma$ in the important special case where the marked bordered surface $\Sigma=(\mathbb{D},\mathbb{M})$ is a disk $\mathbb{D}$ with a finite set~$\mathbb{M}$ of marked points on its boundary. Equivalently, we are considering cluster varieties associated to a quiver of Dynkin type~$A$. We prove the following result, which provides a quantization of the canonical functions $\mathbb{I}_\mathcal{A}(l)$ in this case.

\begin{theorem}
\label{thm:introXproperties}
Let $\Sigma=(\mathbb{D},\mathbb{M})$ be a disk with marked points on its boundary. Then there is a canonical map $\mathbb{I}_{\mathcal{A}}^q:\mathcal{A}_\Sigma(\mathbb{Z}^t)\rightarrow\mathcal{X}_\Sigma^q$ satisfying the following properties:
\begin{enumerate}
\item For any choice of ideal triangulation $T$ of~$\Sigma$, the element $\mathbb{I}_{\mathcal{A}}^q(l)\in\mathcal{X}_\Sigma^q$ can be written as a Laurent polynomial in the variables $X_j$ ($j\in J_T$) with coefficients in $\mathbb{Z}_{\geq0}[q,q^{-1}]$.

\item The expression $\mathbb{I}_{\mathcal{A}}^q(l)$ agrees with the Laurent polynomial $\mathbb{I}_{\mathcal{A}}(l)$ when~$q=1$.

\item Let $*$ be the canonical involutive antiautomorphism of $\mathcal{X}^q$ that fixes each $X_i$ and sends~$q$ to~$q^{-1}$. Then $*\mathbb{I}_{\mathcal{A}}^q(l)=\mathbb{I}_{\mathcal{A}}^q(l)$.

\item The highest term of $\mathbb{I}_{\mathcal{A}}^q(l)$ is
\[
q^{-\sum_{i<j}\varepsilon_{ij}a_ia_j}X_1^{a_1}\dots X_n^{a_n}
\]
where $X_1^{a_1}\dots X_n^{a_n}$ is the highest term of the classical expression $\mathbb{I}_{\mathcal{A}}(l)$.

\item For any $l$,~$l'\in\mathcal{A}_\Sigma(\mathbb{Z}^t)$, we have
\[
\mathbb{I}_{\mathcal{A}}^q(l)\mathbb{I}_{\mathcal{A}}^q(l') = \sum_{l''\in\mathcal{A}_\Sigma(\mathbb{Z}^t)}c^q(l,l';l'')\mathbb{I}_{\mathcal{A}}^q(l'')
\]
where $c^q(l,l';l'')\in\mathbb{Z}[q,q^{-1}]$ and only finitely many terms are nonzero.
\end{enumerate}
\end{theorem}

The existence of the map $\mathbb{I}_{\mathcal{A}}^q$ was conjectured by Fock and Goncharov in~\cite{ensembles}. In physics, it is expected to be closely related to the ``protected spin characters'' introduced by Gaiotto, Moore, and Neitzke for a class of quantum field theories known as the Argyres-Douglas theories associated to a Lie algebra of type $A_1$~\cite{GMN}.

The map $\mathbb{I}_\mathcal{A}^q$ is an analog of the one constructed in~\cite{AK} in the case when the marked bordered surface $\Sigma=(\mathbb{S},\mathbb{M})$ is a punctured surface, that is, when $\mathbb{S}$ has empty boundary and all points of $\mathbb{M}$ lie in the interior of~$\mathbb{S}$. The construction in~\cite{AK} was based on the ``quantum trace map'' introduced by Bonahon and Wong~\cite{BonahonWong} whereas the construction of the present paper is based on the theory of quantum cluster algebras~\cite{BZq}, the notion of a quantum $F$-polynomial from~\cite{Tran}, and the work of Muller on skein algebras~\cite{Muller}. Recent results of L\^{e}~\cite{Le} suggest that these two approaches are in fact equivalent. It should be possible, using these ideas, to generalize the construction to an arbitrary marked bordered surface~$\Sigma$, but we will not pursue such a generalization here.

\subsection{The symplectic double}

In addition to proving Theorem~\ref{thm:introXproperties}, we will show that the duality map $\mathbb{I}_\mathcal{A}^q$ is part of a larger construction involving a cluster variety called the \emph{symplectic double}. This cluster variety was introduced by Fock and Goncharov and plays a key role in their work on quantization of cluster Poisson varieties~\cite{dilog}. As before, one can associate, to a marked bordered surface~$\Sigma$, a certain moduli space of local systems, and this moduli space is birationally equivalent to a symplectic double cluster variety, which we denote by $\mathcal{D}_\Sigma$~\cite{double,Dmoduli}.

The symplectic double is closely related to the cluster $K_2$- and Poisson varieties considered above. In particular, one has a natural embedding $\varphi:\mathcal{A}_\Sigma\rightarrow\mathcal{D}_\Sigma$ and a natural projection $\pi:\mathcal{D}_\Sigma\rightarrow\mathcal{X}_\Sigma$. In previous work~\cite{Dlam}, the author gave a description of the set~$\mathcal{D}_\Sigma(\mathbb{Z}^t)$ of $\mathbb{Z}^t$-valued points of the symplectic double and constructed a canonical map 
\[
\mathbb{I}_{\mathcal{D}}:\mathcal{D}_\Sigma(\mathbb{Z}^t)\rightarrow\mathbb{Q}(\mathcal{D}_\Sigma)
\]
from this set into the field of rational functions on $\mathcal{D}_\Sigma$. This map $\mathbb{I}_{\mathcal{D}}$ is compatible with $\mathbb{I}_\mathcal{A}$ in the sense that there is a commutative diagram 
\[
\xymatrix{ 
\mathcal{A}_\Sigma(\mathbb{Z}^t) \ar[r]^-{\mathbb{I}_\mathcal{A}} \ar[d]_{\varphi} & \mathbb{Q}(\mathcal{X}_\Sigma) \ar[d]^{\pi^*} \\
\mathcal{D}_\Sigma(\mathbb{Z}^t) \ar_-{\mathbb{I}_{\mathcal{D}}}[r] & \mathbb{Q}(\mathcal{D}_\Sigma)
}
\]
where the vertical maps are induced by $\varphi:\mathcal{A}_\Sigma\rightarrow\mathcal{D}_\Sigma$ and $\pi:\mathcal{D}_\Sigma\rightarrow\mathcal{X}_\Sigma$.

In general, the function obtained by applying the map $\mathbb{I}_{\mathcal{D}}$ to a point of $\mathcal{D}_\Sigma(\mathbb{Z}^t)$ is not a regular function on the symplectic double. The natural geometric definition of the map~$\mathbb{I}_{\mathcal{D}}$ produces \emph{rational} rather than regular functions, and in particular, the map $\mathbb{I}_{\mathcal{D}}$ does not provide a canonical basis for regular functions of~$\mathcal{D}_\Sigma$. Nevertheless, these rational expressions have interesting positivity properties. Indeed, for any ideal triangulation $T$ of~$\Sigma$, there is a collection of rational coordinates $B_j:\mathcal{D}_T\dashrightarrow\mathbb{C}^*$ and $X_j:\mathcal{D}_T\dashrightarrow\mathbb{C}^*$ on the symplectic double indexed by the set~$J_T$, and the canonical function $\mathbb{I}_{\mathcal{D}}$ can be written as a ratio of polynomials in these coordinates with nonnegative integer coefficients~\cite{Dlam}.

As its name suggests, the symplectic double has a natural symplectic structure and admits a canonical $q$-deformation~\cite{dilog}. Thus there is a family of noncommutative algebras~$\mathcal{D}_\Sigma^q$ depending on a parameter $q$ such that we recover the function field $\mathbb{Q}(\mathcal{D}_\Sigma)$ by setting $q=1$. For any ideal triangulation~$T$ of~$\Sigma$, the algebra $\mathcal{D}_\Sigma^q$ is isomorphic to the noncommutative fraction field of the algebra generated over $\mathbb{Q}[q,q^{-1}]$ by variables $B_j^{\pm1}$ and $X_j^{\pm1}$ for $j\in J_T$, subject to the relations 
\[
X_iX_j=q^{2\varepsilon_{ij}}X_jX_i, \quad B_iB_j=B_jB_i, \quad X_iB_j=q^{2\delta_{ji}}B_jX_i
\]
for all $i$,~$j\in J_T$.

The second main result of this paper is a quantization of the map~$\mathbb{I}_{\mathcal{D}}$ in the case where $\Sigma$ is a disk with finitely many marked points on its boundary.

\begin{theorem}
Let $\Sigma=(\mathbb{D},\mathbb{M})$ be a disk with marked points on its boundary. Then there is a canonical map $\mathbb{I}_{\mathcal{D}}^q:\mathcal{D}_\Sigma(\mathbb{Z}^t)\rightarrow\mathcal{D}_\Sigma^q$ satisfying the following properties:
\begin{enumerate}
\item For any choice of ideal triangulation $T$ of~$\Sigma$, the element $\mathbb{I}_{\mathcal{D}}^q(l)\in\mathcal{D}_\Sigma^q$ can be written as a rational expression in the variables $B_j$, $X_j$~($j\in J_T$) with coefficients in $\mathbb{Z}_{\geq0}[q,q^{-1}]$.

\item The expression $\mathbb{I}_{\mathcal{D}}^q(l)$ agrees with the rational function $\mathbb{I}_{\mathcal{D}}(l)$ when $q=1$.

\item There is a commutative diagram 
\[
\xymatrix{ 
\mathcal{A}_\Sigma(\mathbb{Z}^t) \ar[r]^-{\mathbb{I}_{\mathcal{A}}^q} \ar[d]_{\varphi} & \mathcal{X}_\Sigma^q \ar[d]^{\pi^q} \\
\mathcal{D}_\Sigma(\mathbb{Z}^t) \ar_-{\mathbb{I}_{\mathcal{D}}^q}[r] & \mathcal{D}_\Sigma^q
}
\]
where $\pi^q$ is an algebra homomorphism that agrees with $\pi^*$ when $q=1$.
\end{enumerate}
\end{theorem}

\subsection{Organization}

The rest of this paper is organized as follows. In Section~\ref{sec:QuantumClusterVarieties}, we define the cluster Poisson variety, the symplectic double, and their $q$-deformations using the quantum dilogarithm function. In Section~\ref{sec:QuantumClusterAlgebras}, we review the tools we need from the theory of quantum cluster algebras. In Section~\ref{sec:TheSkeinAlgebraAndLaminations}, we define the skein algebra and review the results of Muller~\cite{Muller}, which relate skein algebras to quantum cluster algebras. In Section~\ref{sec:DualityMapForTheQuantumPoissonVariety}, we construct the map~$\mathbb{I}_\mathcal{A}^q$ and prove a number of conjectured properties of this map. In Section~\ref{sec:DualityMapForTheQuantumSymplecticDouble}, we define the map~$\mathbb{I}_{\mathcal{D}}^q$ and discuss its relation to~$\mathbb{I}_{\mathcal{A}}^q$. Appendix~\ref{app:DerivationOfMutationFormulas} gives a detailed derivation of the mutation formulas used to define the quantum symplectic double and may be of independent interest to some readers.

\section{Quantum cluster varieties}
\label{sec:QuantumClusterVarieties}

\subsection{Seeds and mutations}

Cluster varieties are geometric objects defined using combinatorial ideas from Fomin and Zelevinsky's theory of cluster algebras~\cite{FZI}. One of the main concepts needed to define cluster varieties is the notion of a seed.

\begin{definition}
\label{def:altseed}
A \emph{seed} $\mathbf{i}=(L,\{e_i\}_{i\in I},\{e_j\}_{j\in J},(\cdot,\cdot))$ is a quadruple where 
\begin{enumerate}
\item $L$ is a lattice with basis $\{e_i\}_{i\in I}$.
\item $\{e_j\}_{j\in J}$ is a subset of the basis.
\item $(\cdot,\cdot)$ is a $\mathbb{Z}$-valued skew-symmetric bilinear form on $L$.
\end{enumerate}
\end{definition}

A basis vector $e_j$ with $j\in J$ is said to be \emph{mutable}, while a basis vector $e_i$ with $i\in I-J$ is said to be \emph{frozen}. Note that if we are given a seed, we can form a skew-symmetric integer matrix with entries $\varepsilon_{ij}=(e_i,e_j)$~($i$,~$j\in I$).

The second main concept that we need to define cluster varieties is the notion of mutation. For any integer $n$, let us write $[n]_+=\max(0,n)$.

\begin{definition}
\label{def:altmutation}
Let $\mathbf{i}=(L,\{e_i\}_{i\in I},\{e_j\}_{j\in J},(\cdot,\cdot))$ be a seed and $e_k$ ($k\in J$) a mutable basis vector. Then we define a new seed $\mathbf{i}'=(L',\{e_i'\}_{i\in I},\{e_j'\}_{j\in J},(\cdot,\cdot)')$ called the seed obtained by \emph{mutation} in the direction of $e_k$. It is given by $L'=L$, $(\cdot,\cdot)'=(\cdot,\cdot)$, and 
\[
e_i'=
\begin{cases}
-e_k & \mbox{if } i=k \\
e_i+[\varepsilon_{ik}]_+e_k & \mbox{if } i\neq k.
\end{cases}
\]
\end{definition}

It is straightforward to calculate the change of the matrix $\varepsilon_{ij}$ under a mutation of seeds.

\begin{proposition}
\label{prop:matrixmutation}
A mutation in the direction~$k$ changes the matrix~$\varepsilon_{ij}$ to the matrix
\[
\varepsilon_{ij}'=
\begin{cases}
-\varepsilon_{ij} & \mbox{if } k\in\{i,j\} \\
\varepsilon_{ij}+\frac{|\varepsilon_{ik}|\varepsilon_{kj}+\varepsilon_{ik}|\varepsilon_{kj}|}{2} & \mbox{if } k\not\in\{i,j\}.
\end{cases}
\]
\end{proposition}

We will denote by $|\mathbf{i}|$ the mutation equivalence class of a seed~$\mathbf{i}$, that is, the collection of all seeds obtained from $\mathbf{i}$ by applying sequences of mutations.

\subsection{The quantum dilogarithm and quantum tori}

To define quantum cluster varieties, we employ the following special function.

\begin{definition}
The \emph{quantum dilogarithm} is the formal infinite product 
\[
\Psi^q(x)=\prod_{k=1}^\infty(1+q^{2k-1}x)^{-1}.
\]
\end{definition}

We will study the quantum dilogarithm as a function on the following algebra of $q$-commuting variables.

\begin{definition}
\label{def:quantumtorus}
Let $L$ be a lattice equipped with a $\mathbb{Z}$-valued skew-symmetric bilinear form~$(\cdot,\cdot)$. Then the \emph{quantum torus algebra} is the noncommutative algebra over $\mathbb{Q}[q,q^{-1}]$ generated by variables $Y_v$ ($v\in L$) subject to the relations 
\[
q^{-(v_1,v_2)}Y_{v_1}Y_{v_2}=Y_{v_1+v_2}.
\]
\end{definition}

This definition allows us to associate to any seed $\mathbf{i}=(L,\{e_i\}_{i\in I},\{e_j\}_{j\in J},(\cdot,\cdot))$, a quantum torus algebra $\mathcal{X}_{\mathbf{i}}^q$. The set $\{e_j\}_{j\in J}$ provides a set of generators $X_j^{\pm1}$ given by $X_j=Y_{e_j}$ for this algebra. They obey the commutation relations 
\[
X_iX_j=q^{2\varepsilon_{ij}}X_jX_i.
\]
This algebra $\mathcal{X}_{\mathbf{i}}^q$ satisfies the Ore condition from ring theory, so we can form its noncommutative fraction field $\Frac({\mathcal{X}}_{\mathbf{i}}^q)$. In addition to associating a quantum torus algebra to every seed, we use the quantum dilogarithm to construct a natural map $\Frac({\mathcal{X}}_{\mathbf{i}'}^q)\rightarrow\Frac({\mathcal{X}}_{\mathbf{i}}^q)$ whenever two seeds $\mathbf{i}$ and~$\mathbf{i}'$ are related by a mutation.

\Needspace*{3\baselineskip}
\begin{definition}\mbox{}
\label{def:quantummutationdef}
\begin{enumerate}
\item The automorphism $\mu_k^\sharp:\Frac({\mathcal{X}}_{\mathbf{i}}^q)\rightarrow\Frac({\mathcal{X}}_{\mathbf{i}}^q)$ is given by conjugation with $\Psi^q(X_k)$:
\[
\mu_k^\sharp=\Ad_{\Psi^q(X_k)}.
\]

\item The isomorphism $\mu_k':\Frac({\mathcal{X}}_{\mathbf{i}'}^q)\rightarrow\Frac({\mathcal{X}}_{\mathbf{i}}^q)$ is induced by the natural lattice map $L'\rightarrow L$.

\item The mutation map $\mu_k^q:\Frac({\mathcal{X}}_{\mathbf{i}'}^q)\rightarrow\Frac({\mathcal{X}}_{\mathbf{i}}^q)$ is the composition $\mu_k^q=\mu_k^\sharp\circ\mu_k'$.
\end{enumerate}
\end{definition}

Note that conjugation by $\Psi^q(X_k)$ produces a priori a formal power series. We will see below that this construction in fact provides a map $\Frac({\mathcal{X}}_{\mathbf{i}'}^q)\rightarrow\Frac({\mathcal{X}}_{\mathbf{i}}^q)$ of skew fields.

\subsection{The quantum double construction}
\label{sec:TheQuantumDoubleConstruction}

As explained in~\cite{dilog}, it is natural to embed the above construction in a larger one. If $\mathbf{i}=(L,\{e_i\}_{i\in I},\{e_j\}_{j\in J},(\cdot,\cdot))$ is any seed, then we can form the ``double'' $L_{\mathcal{D}}=L_{\mathcal{D},\mathbf{i}}$ of the lattice $L$ given by the formula 
\[
L_{\mathcal{D}}=L\oplus L^\vee
\]
where $L^\vee=\Hom(L,\mathbb{Z})$. The basis $\{e_i\}$ for $L$ provides a dual basis $\{f_i\}$ for $L^\vee$, and hence we have a basis $\{e_i,f_i\}$ for $L_{\mathcal{D}}$. Moreover, there is a natural skew-symmetric bilinear form $(\cdot,\cdot)_{\mathcal{D}}$ on~$L_{\mathcal{D}}$ given by the formula 
\[
\left((v_1,\varphi_1),(v_2,\varphi_2)\right)_{\mathcal{D}}=(v_1,v_2)+\varphi_2(v_1)-\varphi_1(v_2).
\]
We can apply the construction of Definition~\ref{def:quantumtorus} to these data to get a quantum torus algebra which we denote $\mathcal{D}_{\mathbf{i}}^q$. If we let $X_i$ and $B_i$ denote the generators associated to the basis elements $e_i$ and $f_i$, respectively, then we have the commutation relations 
\[
X_iX_j=q^{2\varepsilon_{ij}}X_jX_i, \quad B_iB_j=B_jB_i, \quad X_iB_j=q^{2\delta_{ji}}B_jX_i.
\]

We will write $\Frac({\mathcal{D}}_{\mathbf{i}}^q)$ for the (noncommutative) fraction field of~$\mathcal{D}_{\mathbf{i}}^q$. The following notations will be important in the sequel:
\[
\mathbb{B}_k^+=\prod_{i|(e_k,e_i)>0}B_i^{(e_k,e_i)}, \quad \mathbb{B}_k^-=\prod_{i|(e_k,e_i)<0}B_i^{-(e_k,e_i)}, \quad \widehat{X}_i=X_i\prod_jB_j^{(e_i,e_j)}.
\]
One can check using the above relations that the elements $X_k$ and $\widehat{X}_k$ commute. Just as before, we have a natural map $\Frac({\mathcal{D}}_{\mathbf{i}'}^q)\rightarrow\Frac({\mathcal{D}}_{\mathbf{i}}^q)$ whenever $\mathbf{i}$ and $\mathbf{i}'$ are two seeds related by a mutation.

\Needspace*{3\baselineskip}
\begin{definition}\mbox{}
\label{def:doublemutation}
\begin{enumerate}
\item The automorphism $\mu_k^\sharp:\Frac({\mathcal{D}}_{\mathbf{i}}^q)\rightarrow\Frac({\mathcal{D}}_{\mathbf{i}}^q)$ is given by 
\[
\mu_k^\sharp=\Ad_{\Psi^q(X_k)/\Psi^q(\widehat{X}_k)}.
\]

\item The isomorphism $\mu_k':\Frac({\mathcal{D}}_{\mathbf{i}'}^q)\rightarrow\Frac({\mathcal{D}}_{\mathbf{i}}^q)$ is induced by the natural lattice map $L_{\mathcal{D},\mathbf{i}'}\rightarrow L_{\mathcal{D},\mathbf{i}}$.

\item The mutation map $\mu_k^q:\Frac({\mathcal{D}}_{\mathbf{i}'}^q)\rightarrow\Frac({\mathcal{D}}_{\mathbf{i}}^q)$ is the composition $\mu_k^q=\mu_k^\sharp\circ\mu_k'$.
\end{enumerate}
\end{definition}

Notice that the generators $X_i$ span a subalgebra of $\mathcal{D}_{\mathbf{i}}^q$ which is isomorphic to the quantum torus algebra $\mathcal{X}_{\mathbf{i}}^q$ defined previously. Moreover, since $X_k$ and $\widehat{X}_k$ commute, the restriction of~$\mu_k^q$ to this subalgebra coincides with the mutation map from Definition~\ref{def:quantummutationdef}.

Although conjugation by $\Psi^q(X_k)/\Psi^q(\widehat{X}_k)$ produces a priori a formal power series, this construction in fact provides a map $\Frac({\mathcal{D}}_{\mathbf{i}'}^q)\rightarrow\Frac({\mathcal{D}}_{\mathbf{i}}^q)$ of skew fields. Indeed, one has the following formulas for the values of the map $\mu_k^q$ on generators.

\begin{theorem}
\label{thm:transformation}
The map $\mu_k^q$ is given on generators by the formulas 
\[
\mu_k^q(B_i')=
\begin{cases}
(qX_k\mathbb{B}_k^++\mathbb{B}_k^-)B_k^{-1}(1+q^{-1}X_k)^{-1} & \mbox{if } i=k \\
B_i & \mbox{if } i\neq k
\end{cases}
\]
and
\[
\mu_k^q(X_i')=
\begin{cases}
X_i\prod_{p=0}^{|\varepsilon_{ik}|-1}(1+q^{2p+1}X_k) & \mbox{if } \varepsilon_{ik}\leq0 \mbox{ and } i\neq k \\
X_iX_k^{\varepsilon_{ik}}\prod_{p=0}^{\varepsilon_{ik}-1}(X_k+q^{2p+1})^{-1} & \mbox{if } \varepsilon_{ik}\geq0 \mbox{ and } i\neq k \\
X_k^{-1} & \mbox{if } i=k.
\end{cases}
\]
\end{theorem}

The proof of this theorem can be found in Appendix~\ref{app:DerivationOfMutationFormulas}.

\subsection{The classical limit}
\label{sec:TheClassicalLimit}

We have now seen how to associate, to any seed $\mathbf{i}$, a noncommutative algebra $\mathcal{D}_{\mathbf{i}}^q$. In addition, we have seen how to associate, to any pair of seeds $\mathbf{i}$ and $\mathbf{i}'$ related by a mutation, an isomorphism $\Frac({\mathcal{D}}_{\mathbf{i}'}^q)\rightarrow \Frac({\mathcal{D}}_{\mathbf{i}}^q)$ of the corresponding fraction fields.

Note that if we set $q=1$, then $\mathcal{D}_{\mathbf{i}}^q$ becomes the Laurent polynomial ring $\mathbb{Q}[B_i^{\pm1},X_i^{\pm1}]_{i\in J}$. Its spectrum is a split algebraic torus, which we denote by $\mathcal{D}_{\mathbf{i}}$. The formulas of Theorem~\ref{thm:transformation} specialize to 
\[
\mu_k^*(B_i') =
\begin{cases}
\frac{X_k\prod_{j|\varepsilon_{kj>0}}B_j^{\varepsilon_{kj}} + \prod_{j|\varepsilon_{kj<0}}B_j^{-\varepsilon_{kj}}}{(1+X_k)B_k} & \mbox{if } i=k \\
B_i & \mbox{if } i\neq k
\end{cases}
\]
and 
\[
\mu_k^*(X_i')=
\begin{cases}
X_k^{-1} & \mbox{if } i=k \\
X_i{(1+X_k^{-\sgn(\varepsilon_{ik})})}^{-\varepsilon_{ik}} & \mbox{if } i\neq k
\end{cases}
\]
where $\{B_j',X_j'\}_{j\in J}$ are the coordinates on $\mathcal{D}_{\mathbf{i}'}$. These formulas define birational maps $\mathcal{D}_{\mathbf{i}}\dashrightarrow\mathcal{D}_{\mathbf{i}'}$ of tori.

\begin{definition}
The \emph{symplectic double} $\mathcal{D}=\mathcal{D}_{|\mathbf{i}|}$ is the scheme over~$\mathbb{Q}$ obtained by gluing the split algebraic tori $\mathcal{D}_{\mathbf{i}'}$ for all seeds $\mathbf{i}'$ in the mutation equivalence class of~$\mathbf{i}$ using the birational maps above.
\end{definition}

In exactly the same way, the algebra $\mathcal{X}_{\mathbf{i}}^q$ becomes the Laurent polynomial ring $\mathbb{Q}[X_i^{\pm1}]_{i\in J}$ when $q=1$. Its spectrum is a split algebraic torus denoted $\mathcal{X}_{\mathbf{i}}$. As before, we get a birational map $\mathcal{X}_{\mathbf{i}}\dashrightarrow\mathcal{X}_{\mathbf{i}'}$ whenever two seeds $\mathbf{i}$ and $\mathbf{i}'$ are related by a mutation.

\begin{definition}
The \emph{cluster Poisson variety} $\mathcal{X}=\mathcal{X}_{|\mathbf{i}|}$ is the scheme over~$\mathbb{Q}$ obtained by gluing the split algebraic tori $\mathcal{X}_{\mathbf{i}'}$ for all seeds $\mathbf{i}'$  in the mutation equivalence class of~$\mathbf{i}$ using the birational maps above.
\end{definition}

As the names suggest, the cluster Poisson variety has a natural Poisson structure, and the symplectic double carries a natural symplectic form with a compatible Poisson structure. The cluster Poisson variety embeds into the symplectic double as a Lagrangian subspace~\cite{dilog}.

For generic $q$, we can use the formulas of Theorem~\ref{thm:transformation} to glue the algebras $\Frac({\mathcal{X}}_{\mathbf{i}}^q)$ and $\Frac({\mathcal{D}}_{\mathbf{i}}^q)$:
\begin{align*}
\mathcal{X}^q &= \coprod_{\mathbf{i}'\in|\mathbf{i}|}\Frac({\mathcal{X}}_{\mathbf{i}'}^q)/\text{identifications}, \\
\mathcal{D}^q &= \coprod_{\mathbf{i}'\in|\mathbf{i}|}\Frac({\mathcal{D}}_{\mathbf{i}'}^q)/\text{identifications}.
\end{align*}
The sets $\mathcal{X}^q$ and $\mathcal{D}^q$ inherit natural algebra structures and in the classical limit $q=1$ are identified with the function fields of the cluster Poisson variety and symplectic double, respectively. For $q\neq1$, we think of $\mathcal{X}^q$ and $\mathcal{D}^q$ as the function fields of corresponding ``quantum cluster varieties''.

\subsection{The disk case}

From now on, we will write $\Sigma=(\mathbb{D},\mathbb{M})$ for the data of a compact oriented disk $\mathbb{D}$ and a finite set~$\mathbb{M}$ of marked points on the boundary of~$\mathbb{D}$. As we will see in this section, there are quantum cluster varieties naturally associated to~$\Sigma$.

\begin{definition}
An \emph{ideal triangulation} of $\Sigma=(\mathbb{D},\mathbb{M})$ is a triangulation of~$\mathbb{D}$ whose vertices are precisely the points of~$\mathbb{M}$.
\end{definition}

The disk $\Sigma$ admits an ideal triangulation if and only if it has at least three marked points on its boundary. From now on, we will always assume this is the case. The illustration below shows an example of an ideal triangulation of a disk with five marked points.
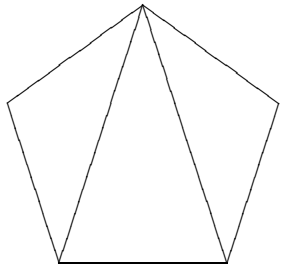
\[
\xy /l1.5pc/:
{\xypolygon5"A"{~:{(-3,0):}}},
{"A5"\PATH~={**@{-}}'"A2"'"A4"},
\endxy
\]

An edge of an ideal triangulation of $\Sigma=(\mathbb{D},\mathbb{M})$ is said to be \emph{external} if it lies along the boundary of $\mathbb{D}$, connecting adjacent marked points, and is said to be \emph{internal} otherwise. For a given ideal triangulation $T$, we will write $I=I_T$ for the set of edges of $T$ and $J=J_T$ for the set of internal edges.

\begin{definition}
Let $T$ be an ideal triangulation of $\Sigma$. Then we define a skew-symmetric matrix $\varepsilon_{ij}$~($i$,~$j\in I$) by 
\[
\varepsilon_{ij}=
\begin{cases}
-1 & \mbox{if $i$, $j$ share a vertex and $i$ is immediately clockwise to $j$} \\
1 & \mbox{if $i$, $j$ share a vertex and $j$ is immediately clockwise to $i$} \\
0 & \mbox{otherwise}.
\end{cases}
\]
\end{definition}

Given an ideal triangulation $T$, we can consider the associated lattice 
\[
L=\mathbb{Z}[I].
\]
This lattice has a basis $\{e_i\}$ given by $e_i=\{i\}$ for $i\in I$ and a $\mathbb{Z}$-valued skew-symmetric bilinear form $(\cdot,\cdot)$ given on basis elements by 
\[
(e_i,e_j)=\varepsilon_{ij}.
\]
Thus we associate a seed to any ideal triangulation.

\begin{definition}
If $k$ is an internal edge of an ideal triangulation $T$, then a \emph{flip} at~$k$ is the transformation of~$T$ that removes $k$ and replaces it by the unique different edge that, together with the remaining edges, forms an ideal triangulation:
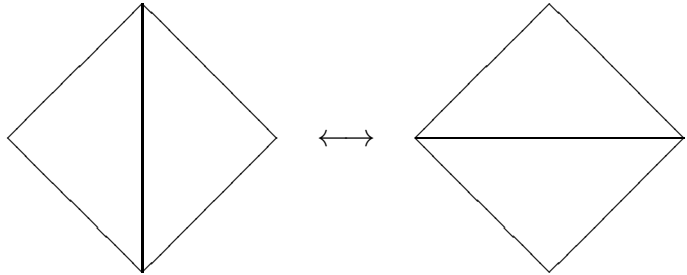
\[
\xy /l1.5pc/:
{\xypolygon4"A"{~:{(2,2):}}},
{"A1"\PATH~={**@{-}}'"A3"},
\endxy
\quad
\longleftrightarrow
\quad
\xy /l1.5pc/:
{\xypolygon4"A"{~:{(2,2):}}},
{"A2"\PATH~={**@{-}}'"A4"}
\endxy
\]
\end{definition}

It is a fact that any two ideal triangulations are related by a sequence of flips. It is therefore natural to ask how the matrix $\varepsilon_{ij}$ changes when we perform a flip at some edge~$k$. To answer this question, first note that there is a natural bijection between the edges of an ideal triangulation and the edges of the triangulation obtained by a flip at some edge. If we use this bijection to identify edges of the flipped triangulation with the set~$I$, then we have the following result.

\begin{proposition}
A flip at an edge~$k$ of an ideal triangulation changes the matrix~$\varepsilon_{ij}$ to the matrix
\[
\varepsilon_{ij}'=
\begin{cases}
-\varepsilon_{ij} & \mbox{if } k\in\{i,j\} \\
\varepsilon_{ij}+\frac{|\varepsilon_{ik}|\varepsilon_{kj}+\varepsilon_{ik}|\varepsilon_{kj}|}{2} & \mbox{if } k\not\in\{i,j\}.
\end{cases}
\]
\end{proposition}

Thus we see that the new matrix $\varepsilon_{ij}'$ that we get after performing a flip at the edge~$k$ is the same as the matrix that we get by mutating the seed associated to $T$. It follows that the quantum cluster varieties defined above can be \emph{canonically} associated to a disk with marked points so that each seed corresponds to an ideal triangulation of the disk.

\begin{definition}
We will write $\mathcal{D}_\Sigma^q$ for the algebra $\mathcal{D}^q$ associated to $\Sigma$ in this way and $\mathcal{X}_\Sigma^q$ for the algebra $\mathcal{X}^q$ associated to~$\Sigma$.
\end{definition}

\section{Quantum cluster algebras}
\label{sec:QuantumClusterAlgebras}

\subsection{General theory of quantum cluster algebras}

Here we review the theory of quantum cluster algebras, following~\cite{BZq,Tran}. Throughout this section, $m$ and $n$ will be positive integers with $m\geq n$.

\begin{definition}
Let $k\in\{1,\dots,n\}$. We say that an $m\times n$ matrix $\mathbf{B}'=(b_{ij}')$ is obtained from an $m\times n$ matrix $\mathbf{B}=(b_{ij})$ by \emph{matrix mutation} in the direction $k$ if the entries of $\mathbf{B}'$ are given by 
\[
b_{ij}'=
\begin{cases}
-b_{ij} & \mbox{if } k\in\{i,j\} \\
b_{ij}+\frac{|b_{ik}|b_{kj}+b_{ik}|b_{kj}|}{2} & \mbox{if } k\not\in\{i,j\}.
\end{cases}
\]
In this case, we write $\mu_k(\mathbf{B})=\mathbf{B}'$.
\end{definition}

\begin{definition}
\label{def:compatibility}
Let $\mathbf{B}=(b_{ij})$ be an $m\times n$ integer matrix, and let $\Lambda=(\lambda_{ij})$ be a skew-symmetric $m\times m$ integer matrix. We say that the pair $(\Lambda,\mathbf{B})$ is \emph{compatible} if for each $j\in\{1,\dots,n\}$ and $i\in\{1,\dots,m\}$, we have
\[
\sum_{k=1}^mb_{kj}\lambda_{ki}=\delta_{ij}d_j
\]
for some positive integers $d_j$~($j\in\{1,\dots,n\}$). Equivalently, the product $\mathbf{B}^t\Lambda$ equals the $n\times m$ matrix $(D|0)$ where $D$ is the $n\times n$ diagonal matrix with diagonal entries $d_1,\dots,d_n$.
\end{definition}

Let $k\in\{1,\dots,n\}$ and choose a sign $\epsilon\in\{\pm1\}$. Denote by $E_\epsilon$ the $m\times m$ matrix with entries given by 
\[
e_{ij}=
\begin{cases}
\delta_{ij} & \mbox{if } j\neq k \\
-1 & \mbox{if } i=j=k \\
\max(0,-\epsilon b_{ik}) & \mbox{if } i\neq j=k
\end{cases}
\]
and set 
\[
\Lambda'=E_\epsilon^t\Lambda E_\epsilon.
\]

\begin{proposition}[\cite{BZq}, Proposition~3.4]
The matrix $\Lambda'$ is skew-symmetric and independent of the sign $\epsilon$. Moreover, $(\Lambda',\mu_k(\mathbf{B}))$ is a compatible pair.
\end{proposition}

\begin{definition}[\cite{BZq}, Definition~3.5]
Let $(\Lambda,\mathbf{B})$ be a compatible pair and let $k\in\{1,\dots,n\}$. We say that the pair $(\Lambda',\mu_k(\mathbf{B}))$ is obtained from $(\Lambda,\mathbf{B})$ by \emph{mutation} in the direction~$k$ and write $\mu_k(\Lambda,\mathbf{B})=(\Lambda',\mu_k(\mathbf{B}))$.
\end{definition}

Let $L$ be a lattice of rank $m$ equipped with a skew-symmetric bilinear form $\Lambda:L\times L\rightarrow\mathbb{Z}$. Let $\omega$ be a formal variable. We can associate to these data a quantum torus algebra~$\mathcal{T}$. It is generated over $\mathbb{Z}[\omega,\omega^{-1}]$ by variables $A^v$~($v\in L$) subject to the commutation relations 
\[
A^{v_1}A^{v_2}=\omega^{-\Lambda(v_1,v_2)}A^{v_1+v_2}.
\]
In the literature on quantum cluster algebras, this quantum torus algebra is typically called a \emph{based quantum torus}, and the parameter is denoted $q^{-1/2}$, rather than $\omega$. (See~\cite{BZq,Muller,Tran} for example.) This quantum torus algebra has a noncommutative fraction field which we denote~$\mathcal{F}$.

\begin{definition}
A \emph{toric frame} in $\mathcal{F}$ is a mapping $M:\mathbb{Z}^m\rightarrow\mathcal{F}-\{0\}$ of the form 
\[
M(v)=\phi(A^{\eta(v)})
\]
where $\phi$ is a $\mathbb{Q}(\omega^{-1})$-algebra automorphism of $\mathcal{F}$ and $\eta:\mathbb{Z}^m\rightarrow L$ is an isomorphism of lattices.
\end{definition}

Note that the image $M(\mathbb{Z}^m)$ of a toric frame is a basis for an isomorphic copy of~$\mathcal{T}$ in~$\mathcal{F}$. We have the relations 
\begin{align*}
M(v_1)M(v_2) &= \omega^{-\Lambda_M(v_1,v_2)}M(v_1+v_2), \\
M(v_1)M(v_2) &= \omega^{-2\Lambda_M(v_1,v_2)}M(v_2)M(v_1), \\
M(v)^{-1} &= M(-v), \\
M(0) &=1
\end{align*}
where the form $\Lambda_M$ on $\mathbb{Z}^m$ is obtained from $\Lambda$ using the isomorphism $\eta$.

\begin{definition}
A \emph{quantum seed} is a pair $(M,\mathbf{B})$ where $M$ is a toric frame in~$\mathcal{F}$ and $\mathbf{B}$ is an $m\times n$ integer matrix such that $(\Lambda_M,\mathbf{B})$ is a compatible pair where we view $\Lambda_M$ as a matrix using the standard basis for~$\mathbb{Z}^m$.
\end{definition}

\begin{definition}
\label{def:mutatedtoricframe}
Let $(M,\mathbf{B})$ be a quantum seed and write $\mathbf{B}=(b_{ij})$. For any index $k\in\{1,\dots,n\}$ and any sign $\epsilon\in\{\pm1\}$, we define a mapping $M':\mathbb{Z}^m\rightarrow\mathcal{F}-\{0\}$ by the formulas 
\begin{align*}
M'(v) &= \sum_{p=0}^{v_k}\binom{v_k}{p}_{\omega^{-d_k}}M(E_\epsilon v+\epsilon p b^k), \\
M'(-v) &= M'(v)^{-1}
\end{align*}
where $v=(v_1,\dots,v_m)\in\mathbb{Z}^m$ is such that $v_k\geq0$ and $b^k$ denotes the $k$th column of $\mathbf{B}$. Here the $t$-binomial coefficient is given by 
\[
\binom{r}{p}_t=\frac{(t^r-t^{-r})\dots(t^{r-p+1}-t^{-r+p-1})}{(t^p-t^{-p})\dots(t-t^{-1})}.
\]
\end{definition}

\begin{proposition}[\cite{BZq}, Proposition~4.7]
The mapping $M'$ satisfies the following properties:
\begin{enumerate}
\item The mapping $M'$ is a toric frame which is independent of the sign~$\epsilon$.

\item The pair $(\Lambda_{M'},\mu_k(\mathbf{B}))$ is compatible and obtained from the pair $(\Lambda_M,\mathbf{B})$ by mutation in the direction~$k$.

\item The pair $(M',\mu_k(\mathbf{B}))$ is a quantum seed.
\end{enumerate}
\end{proposition}

\begin{definition}
Let $(M,\mathbf{B})$ be a quantum seed, and let $k\in\{1,\dots,n\}$. Let $M'$ be the mapping from Definition~\ref{def:mutatedtoricframe}, and let $\mathbf{B}'=\mu_k(\mathbf{B})$. Then we say that the quantum seed $(M',\mathbf{B}')$ is obtained from $(M,\mathbf{B})$ by \emph{mutation} in the direction~$k$.
\end{definition}

\begin{proposition}[\cite{BZq}, Proposition~4.9]
\label{prop:exchangetoricframe}
Let $(M,\mathbf{B})$ be a quantum seed, and suppose that $(M',\mathbf{B}')$ is obtained from $(M,\mathbf{B})$ by mutation in the direction~$k$. Then 
\[
M'(\mathbf{e}_k)=M\big(-\mathbf{e}_k+\sum_{i=1}^m[b_{ik}]_+\mathbf{e}_i\big) + M\big(-\mathbf{e}_k+\sum_{i=1}^m[-b_{ik}]_+\mathbf{e}_i\big)
\]
and $M'(\mathbf{e}_i)=M(\mathbf{e}_i)$ for $i\neq k$.
\end{proposition}

\begin{definition}
Denote by $\mathbb{T}_n$ an $n$-regular tree with edges labeled by the numbers $1,\dots,n$ in such a way that the $n$ edges emanating from any vertex have distinct labels. A \emph{quantum cluster pattern} is an assignment of a quantum seed $\mathbf{i}_t=(M_t,\mathbf{B}_t)$ to each vertex $t\in\mathbb{T}_n$ so that if $t$ and $t'$ are vertices connected by an edge labeled~$k$, then $\mathbf{i}_{t'}$ is obtained from $\mathbf{i}_t$ by a mutation in the direction~$k$.
\end{definition}

Given a quantum cluster pattern, let us define $A_{i;t}=M_t(\mathbf{e}_i)$. For $i\in\{n+1,\dots,m\}$, we have $A_{i;t}=A_{i;t'}$ for all $t$,~$t'\in\mathbb{T}_n$, so we may omit one of the subscripts and write $A_i=A_{i;t}$ for all $t\in\mathbb{T}_n$. Let 
\[
\mathcal{S}=\{A_{i;t}:i\in\{1,\dots,n\},t\in\mathbb{T}_n\}.
\]

\begin{definition}[\cite{BZq}, Definition~4.12]
Given a quantum cluster pattern $t\mapsto(M_t,\mathbf{B}_t)$, the associated \emph{quantum cluster algebra} $\mathcal{A}$ is the $\mathbb{Z}[\omega^{\pm1},A_{n+1}^{\pm1},\dots,A_m^{\pm1}]$-subalgebra of the ambient skew-field~$\mathcal{F}$ generated by elements of $\mathcal{S}$.
\end{definition}

\subsection{Quantum $F$-polynomials}
\label{sec:QuantumFPolynomials}

One of the important tools that we apply in our construction of the map $\mathbb{I}_{\mathcal{D}}^q$ is the notion of a quantum $F$-polynomial from~\cite{Tran}. This extends Fomin and Zelevinsky's notion of $F$-polynomial~\cite{FZIV} to the noncommutative setting and allows us to express a generator $A_{j;t}$ of a quantum cluster algebra in terms of the generators associated with an initial seed.

\begin{theorem}[\cite{Tran}, Theorem~5.3]
\label{thm:quantumF}
Let $(M_0,\mathbf{B}_0)$ be an initial quantum seed in a quantum cluster algebra $\mathcal{A}$ and write $\mathbf{B}_0=(b_{ij})$. Then there exists an integer $\lambda_{j;t}\in\mathbb{Z}$ and a polynomial~$F_{j;t}$ in the variables 
\[
Y_j=M_0\big(\sum_i b_{ij}\mathbf{e}_i\big)
\]
with coefficients in $\mathbb{Z}[\omega,\omega^{-1}]$ such that the cluster variable $A_{j;t}\in\mathcal{A}$ is given by 
\[
A_{j;t}=\omega^{\lambda_{j;t}}F_{j;t}\cdot M_0(\mathbf{g}_{j;t})
\]
where $\mathbf{g}_{j;t}$ is an integer vector called the extended $\mathbf{g}$-vector of~$A_{j;t}$.
\end{theorem}

The polynomial $F_{j;t}$ appearing in the theorem is known as a \emph{quantum $F$-polynomial}. In this paper, we restrict attention to quantum cluster algebras where the matrix $\mathbf{B}$ arises from a triangulation of a disk. For such algebras, we have the following refinement of Theorem~\ref{thm:quantumF}.

\begin{corollary}
\label{cor:Fpositivity}
Let $\mathcal{A}$ be a quantum cluster algebra associated to a disk, and suppose the matrix~$D$ appearing in the compatibility condition of Definition~\ref{def:compatibility} is four times the identity. Then there exists a polynomial $F_{j;t}$ in the variables $Y_1,\dots,Y_n$ with coefficients in $\mathbb{Z}_{\geq0}[\omega^4,\omega^{-4}]$ such that the cluster variable $A_{j;t}\in\mathcal{A}$ is given by 
\[
A_{j;t}=F_{j;t}\cdot M_0(\mathbf{g}_{j;t}).
\]
\end{corollary}

\begin{proof}
In the disk case, it is known that each classical $F$-polynomial has nonzero constant term (for example by~\cite{MSW1}). Hence, by~\cite{Tran}, Theorem~6.1, we have $\lambda_{j;t}=0$ in Theorem~\ref{thm:quantumF}. By~\cite{Tran}, Theorem~7.4, we know that the coefficients of $F_{j;t}$ are Laurent polynomials in $\omega^4$ with positive integral coefficients.
\end{proof}

\section{The skein algebra and laminations}
\label{sec:TheSkeinAlgebraAndLaminations}

\subsection{Definition of the skein algebra}

The quantum cluster algebras that we consider in this paper all arise from the skein algebra of~$\Sigma$. Here we will review the the general theory of skein algebras, following~\cite{Muller}.

\begin{definition}
\label{def:curve}
By a \emph{curve} in $\Sigma=(\mathbb{D},\mathbb{M})$, we mean an immersion $C\rightarrow\mathbb{D}$ of a compact, connected, one-dimensional manifold $C$ with (possibly empty) boundary into $\mathbb{D}$. We require that any boundary points of~$C$ map to marked points in $\mathbb{M}$ and no point in the interior of~$C$ maps to a marked point. By a \emph{homotopy} of two curves $\alpha$ and $\beta$, we mean a homotopy of $\alpha$ and $\beta$ within the class of such curves. Two curves are said to be \emph{homotopic} if they can be related by homotopy and orientation-reversal.
\end{definition}

\begin{definition}
A \emph{multicurve} on~$\Sigma$ is an unordered finite set of curves on~$\Sigma$ which may contain duplicates. Two multicurves on~$\Sigma$ are \emph{homotopic} if there is a bijection between their constituent curves which takes each curve to another which is homotopic in the sense of the previous definition.
\end{definition}

\begin{definition}
Let $\Sigma=(\mathbb{D},\mathbb{M})$. By a \emph{strand} of a multicurve in~$\Sigma$ near a point $p\in\mathbb{D}$, we mean a component of the intersection of the multicurve with a small disk around~$p$. A \emph{framed link} is a multicurve such that each intersection of strands is transverse and 
\begin{enumerate}
\item At each crossing, there is an ordering of the strands.
\item At each marked point, there is an equivalence relation on the strands and an ordering on equivalence classes of strands.
\end{enumerate}
By a \emph{homotopy} of framed links, we mean a homotopy through the class of multicurves with transverse intersections where the crossing data are not changed.
\end{definition}

When drawing pictures of framed links, we indicate the ordering of strands at a transverse intersection or marked point by making one strand pass ``over'' the other:
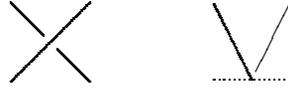
\[
\xygraph{
    !{0;/r2.5pc/:}
    [u(0.5)]!{\xoverv}
}
\qquad
\qquad
\xy 
(0,5)*{}; (5,-5)*{} **\crv{(0,5)&(5,-5)};
\POS?(1)*{\hole}="x"; 
(10,5)*{}; "x" **\crv{}; 
(0,-5)*{}; (10,-5)*{} **\dir{.};
(5,-5.5)*{\sbullet};
\endxy
\]
(In the second of these pictures, the dotted line indicates a portion of $\partial\mathbb{D}$ containing a marked point.)
If two strands are identified by the equivalence relation at a marked point, we indicate this as in the following picture: 
\[
\xy 
(0,5)*{}; (-5,-5)*{} **\crv{(0,5)&(-5,-5)};
(-10,5)*{}; (-5,-5)*{} **\crv{}; 
(0,-5)*{}; (-10,-5)*{} **\dir{.};
(-5,-5.5)*{\sbullet};
\endxy
\]

\begin{definition}
A multicurve in $\Sigma=(\mathbb{D},\mathbb{M})$ with transverse intersections is said to be \emph{simple} if it has no intersections in the interior of~$\mathbb{D}$ and no contractible curves. Note that any simple multicurve can be regarded as a framed link in which, for every marked point~$p$, the strands incident to~$p$ are identified by the equivalence relation.
\end{definition}

\begin{definition}
\label{def:skein}
Let us write $\mathcal{K}(\Sigma)$ for the free $\mathbb{Z}[\omega,\omega^{-1}]$-module generated by equivalence classes of framed links in~$\Sigma=(\mathbb{D},\mathbb{M})$. The \emph{skein module} $\mathrm{Sk}_\omega(\Sigma)$ is defined as the quotient of $\mathcal{K}(\Sigma)$ by the following local relations. In each of these expressions, we depict the portion of a framed link over a small disk in~$\mathbb{D}$. The framed links appearing in a given relation are assumed to be identical to each other outside of the small disk. In the last two relations, the dotted line segment represents a portion of $\partial\mathbb{D}$. In these pictures, there may be additional undrawn curves ending at marked points, provided their order with respect to the drawn curves and each other does not change.
\begin{align*}
\xygraph{
    !{0;/r2.5pc/:}
    [u(0.5)]!{\xoverv}
}
\quad
&=
\quad
\omega^{-2}
\quad
\xygraph{
    !{0;/r2.5pc/:}
    [u(0.5)]!{\xunoverv}
}
\quad
+
\quad
\omega^2
\quad
\xygraph{
    !{0;/r2.5pc/:}
    [u(0.5)]!{\xunoverh}
} \\
\xygraph{
    !{0;/r2.5pc/:}
    !{\vcap-}
    !{\vcap}
}
\quad
&=
\quad
-(\omega^4+\omega^{-4}) \\
\omega
\quad
{\xy 
(0,5)*{}; (-5,-5)*{} **\crv{(0,5)&(-5,-5)};
\POS?(1)*{\hole}="x"; 
(-10,5)*{}; "x" **\crv{}; 
(0,-5)*{}; (-10,-5)*{} **\dir{.};
(-5,-5.5)*{\sbullet};
\endxy}
\quad
&=
\quad
{\xy 
(0,5)*{}; (-5,-5)*{} **\crv{(0,5)&(-5,-5)};
(-10,5)*{}; (-5,-5)*{} **\crv{}; 
(0,-5)*{}; (-10,-5)*{} **\dir{.};
(-5,-5.5)*{\sbullet};
\endxy}
\quad
=
\quad
\omega^{-1}
\quad
{\xy 
(0,5)*{}; (5,-5)*{} **\crv{(0,5)&(5,-5)};
\POS?(1)*{\hole}="x"; 
(10,5)*{}; "x" **\crv{}; 
(0,-5)*{}; (10,-5)*{} **\dir{.};
(5,-5.5)*{\sbullet};
\endxy} \\
{\xy 
(-2,2)*{}; (-5,-5)*{} **\crv{(-2,2)&(-5,-5)};
\POS?(1)*{\hole}="x"; 
(-8,2)*{}; "x" **\crv{}; 
(-2,2)*{}; (-8,2)*{} **\crv{(0,8)&(-10,8)};
(0,-5)*{}; (-10,-5)*{} **\dir{.};
(-5,-5.5)*{\sbullet};
\endxy}
\quad
&=
\quad
{\xy 
(-2,2)*{}; (-5,-5)*{} **\crv{(-2,2)&(-5,-5)};
(-8,2)*{}; (-5,-5)*{} **\crv{}; 
(-2,2)*{}; (-8,2)*{} **\crv{(0,8)&(-10,8)};
(0,-5)*{}; (-10,-5)*{} **\dir{.};
(-5,-5.5)*{\sbullet};
\endxy}
\quad
=
\quad
{\xy 
(2,2)*{}; (5,-5)*{} **\crv{(2,2)&(5,-5)};
\POS?(1)*{\hole}="x"; 
(8,2)*{}; "x" **\crv{}; 
(2,2)*{}; (8,2)*{} **\crv{(0,8)&(10,8)};
(0,-5)*{}; (10,-5)*{} **\dir{.};
(5,-5.5)*{\sbullet};
\endxy}
\quad
=
\quad 0
\end{align*}
If $K$ is a framed link in $\Sigma$, then the class of $K$ in $\mathrm{Sk}_\omega(\Sigma)$ will be denoted $[K]$.
\end{definition}

Suppose $K$ and $L$ are two links such that the union of the underlying multicurves has transverse intersections. Then the \emph{superposition} $K\cdot L$ is the framed link whose underlying multicurve is the union of the underlying multicurves of $K$ and $L$, with each strand of $K$ crossing over each strand of $L$ and all other crossings ordered as in~$K$ and~$L$.

\begin{proposition}[\cite{Muller}, Proposition~3.5]
\label{prop:superpositionhomotopy}
$[K\cdot L]$ depends only on the homotopy classes of~$K$ and~$L$.
\end{proposition}

\begin{definition}
For any framed links $K$ and $L$, choose homotopic links $K'$ and $L'$ such that the union of the multicurves underlying $K'$ and $L'$ is transverse. Then the \emph{superposition product} is defined by 
\[
[K][L]\coloneqq[K'\cdot L'].
\]
This extends to a product on $\mathrm{Sk}_\omega(\Sigma)$ by bilinearity.
\end{definition}

Note that the superposition product is well defined and independent of the choice of~$K'$ and~$L'$ by Proposition~\ref{prop:superpositionhomotopy}.

\subsection{Relation to quantum cluster algebras}

We conclude this section by describing the relation, first discovered by Muller in~\cite{Muller}, between skein algebras and quantum cluster algebras. As usual, we take $\Sigma=(\mathbb{D},\mathbb{M})$ to be a disk with finitely many marked points on its boundary.

\begin{definition}
Let $T$ be an ideal triangulation of $\Sigma$. For any $i\in I$ and $j\in J$, we define 
\[
b_{ij}=
\begin{cases}
1 & \mbox{if $i$, $j$ share a vertex and $i$ is immediately clockwise to $j$} \\
-1 & \mbox{if $i$, $j$ share a vertex and $j$ is immediately clockwise to $i$} \\
0 & \mbox{otherwise}.
\end{cases}
\]
These are entries of an $|I|\times|J|$ matrix which we denote $\mathbf{B}_T$.
\end{definition}

\begin{definition}
Let $T$ be an ideal triangulation of $\Sigma$. For $i$,~$j\in I$, we define 
\[
\lambda_{ij}=
\begin{cases}
1 & \mbox{if $i$, $j$ share a vertex and $i$ is clockwise to $j$} \\
-1 & \mbox{if $i$, $j$ share a vertex and $j$ is clockwise to $i$} \\
0 & \mbox{otherwise}.
\end{cases}
\]
These are entries of a skew-symmetric $|I|\times|I|$ matrix which we denote $\Lambda_T$. We will use the same notation $\Lambda_T$ for the associated skew-symmetric bilinear form $\mathbb{Z}^{I}\times\mathbb{Z}^{I}\rightarrow\mathbb{Z}$.
\end{definition}

Note that the word ``clockwise'' in the definition of $\Lambda_T$ does not necessarily mean ``immediately clockwise''. This definition makes sense because we work on a disk and therefore $i$ cannot be both clockwise and counterclockwise to~$j$.

\begin{proposition}[\cite{Muller}, Proposition~7.8]
\label{prop:compatibility}
The matrices $\Lambda_T$ and $\mathbf{B}_T$ satisfy the compatibility condition 
\[
\sum_kb_{kj}\lambda_{ki}=4\delta_{ij}.
\]
\end{proposition}

From now on, we will write $\mathcal{F}$ for the skew-field of fractions of the skein algebra $\mathrm{Sk}_\omega(\Sigma)$. By Theorem~6.14 of~\cite{Muller}, it is isomorphic to the skew field of fractions of the quantum torus associated with the form $\Lambda_T$ for any ideal triangulation~$T$.

\begin{definition}
For any ideal triangulation $T$ and any vector $v=(v_1,\dots,v_m)\in\mathbb{Z}_{\geq0}^{I}$, we will write $T^v$ for the simple multicurve having $v_i$-many curves homotopic to $i\in I$ and no other components. The corresponding class~$[T^v]$ is called a \emph{monomial} in the triangulation~$T$. More generally, we write 
\[
[T^{u'-u}]=\omega^{-\Lambda_T(u,u')}[T^u]^{-1}[T^{u'}].
\]
This is well defined and provides a map $M_T:\mathbb{Z}^{I}\rightarrow\mathcal{F}-\{0\}$ given by $M_T(v)=[T^v]$.
\end{definition}

With these definitions, the pair $(\mathbf{B}_T,M_T)$ is a quantum seed and $\Lambda_T$ is the compatibility matrix associated to the toric frame~$M_T$.

\begin{proposition}[\cite{Muller}, Theorem~7.9]
\label{prop:quantumflip}
Let $T$ be an ideal triangulation of $S$, and let $T'$ be the ideal triangulation obtained from~$T$ by performing a flip of the edge~$k$. Then the quantum seed $(\mathbf{B}_{T'},M_{T'})$ is obtained from $(\mathbf{B}_T,M_T)$ by a mutation in the direction~$k$.
\end{proposition}

It follows from Proposition~\ref{prop:quantumflip} that there is a quantum cluster algebra $\mathcal{A}$ canonically associated to~$\Sigma$. This algebra is generated by the cluster variables inside of the skew-field of fractions of the skein algebra.

\subsection{Laminations}
\label{sec:Laminations}

In order to construct the maps $\mathbb{I}_{\mathcal{A}}^q$ and $\mathbb{I}_{\mathcal{D}}^q$, we need to define the sets $\mathcal{A}_\Sigma(\mathbb{Z}^t)$ and $\mathcal{D}_\Sigma(\mathbb{Z}^t)$. As explained in the introduction, these parametrize points of~$\mathcal{A}_\Sigma$ and~$\mathcal{D}_\Sigma$ valued in the semifield~$\mathbb{Z}^t$ of tropical integers. Further explanation can be found in~\cite{Dlam,Dmoduli,dual,double}.

\begin{definition}
Let $\Sigma=(\mathbb{D},\mathbb{M})$ be a marked bordered surface consisting of a disk~$\mathbb{D}$ with a finite set $\mathbb{M}$ of marked points on its boundary. An \emph{$\mathcal{A}$-lamination} is defined as a collection of mutually nonintersecting and nonhomotopic curves on~$\Sigma$ connecting points of~$\mathbb{M}$, each equipped with a nonzero integral weight and subject to the following conditions:
\begin{enumerate}
\item If a curve is not homotopic to a segment on~$\partial\mathbb{D}$ connecting adjacent marked points, then its weight must be positive.
\item The total weight of all curves incident to a given marked point is zero.
\end{enumerate}
Two $\mathcal{A}$-laminations are considered to be equivalent if there is a bijection between their constituent curves which sends each curve to a homotopic one with the same weight. We write $\mathcal{A}_\Sigma(\mathbb{Z}^t)$ for the set of equivalence classes of $\mathcal{A}$-laminations.
\end{definition}

The diagram below shows a typical example of an $\mathcal{A}$-lamination associated to a disk with five marked points.
\[
\xy /l1.5pc/:
{\xypolygon5"A"{~:{(-3,0):}~>{.}}},
{\xypolygon5"B"{~:{(-3.5,0):}~>{}}},
"A2";"A4" **\crv{(1,-1) & (2,1)};
"A4";"A5" **\crv{(2,2) & (0,2)};
"A5";"A1" **\crv{(-0.5,2) & (-1,0)};
"A1";"A2" **\crv{(-1,-1) & (1,-2.5)};
"A1"*{\bullet}, 
"A2"*{\bullet}, 
"A3"*{\bullet}, 
"A4"*{\bullet}, 
"A5"*{\bullet}, 
(2,-0.5)*{+1}, 
(-0.5,0.5)*{+1}, 
(1,1.8)*{-1}, 
(0,-1.25)*{-1}, 
"B1"*{p_5}, 
"B2"*{p_1}, 
"B3"*{p_2}, 
"B4"*{p_3}, 
"B5"*{p_4}, 
\endxy
\]

In the next definition, we let $\Sigma=(\mathbb{D},\mathbb{M})$ be a disk with finitely many marked points on its boundary and let $\Sigma^\circ=(\mathbb{D}^\circ,\mathbb{M}^\circ)$ where $\mathbb{D}^\circ$ is the disk~$\mathbb{D}$ equipped with the opposite orientation and $\mathbb{M}^\circ$ is the image of~$\mathbb{M}$ under the tautological map $\mathbb{D}\rightarrow\mathbb{D}^\circ$. 

\begin{definition}
Let $\Sigma$ and $\Sigma^\circ$ be as above. A \emph{$\mathcal{D}$-lamination} is defined as a collection of mutually nonintersecting and nonhomotopic curves on~$\Sigma$ and~$\Sigma^\circ$ connecting the marked points, each equipped with a positive integral weight and subject to the following conditions:
\begin{enumerate}
\item The total weight of curves on~$\Sigma$ incident to a given marked point in~$\mathbb{M}$ is the same as the total weight of the curves on~$\Sigma^\circ$ incident to the corresponding point in~$\mathbb{M}^\circ$.
\item The collection does not include a curve connecting adjacent points of~$\mathbb{M}$ together with a curve connecting the corresponding points of~$\mathbb{M}^\circ$.
\end{enumerate}
Two $\mathcal{D}$-laminations are considered to be equivalent if there is a bijection between their constituent curves which sends any curve on~$\Sigma$ (respectively, $\Sigma^\circ$) to a homotopic curve on~$\Sigma$ (respectively, $\Sigma^\circ$) with the same weight. We write $\mathcal{D}_\Sigma(\mathbb{Z}^t)$ for the set of equivalence classes of $\mathcal{D}$-laminations.
\end{definition}

The diagrams below show a typical example of an $\mathcal{D}$-lamination associated to a disk with five marked points. Here we label a point of~$\mathbb{M}$ and the corresponding point of~$\mathbb{M}^\circ$ by the same symbol.
\[
\xy /l1.5pc/:
{\xypolygon5"A"{~:{(-3,0):}~>{.}}},
{\xypolygon5"B"{~:{(-3.5,0):}~>{}}},
"A4";"A5" **\crv{(2,2) & (0,2)};
"A1";"A2" **\crv{(-1,-1) & (1,-2.5)};
"A1"*{\bullet}, 
"A2"*{\bullet}, 
"A3"*{\bullet}, 
"A4"*{\bullet}, 
"A5"*{\bullet}, 
(1,1.8)*{+1}, 
(0,-1.25)*{+1}, 
"B1"*{p_5}, 
"B2"*{p_1}, 
"B3"*{p_2}, 
"B4"*{p_3}, 
"B5"*{p_4}, 
(5.5,0)*{\Sigma:}, 
\endxy
\qquad
\xy /l1.5pc/:
{\xypolygon5"A"{~:{(-3,0):}~>{.}}},
{\xypolygon5"B"{~:{(-3.5,0):}~>{}}},
"A2";"A5" **\crv{(1,-1) & (0,1)};
"A4";"A3" **\crv{(2.5,2) & (3,0)};
"A1"*{\bullet}, 
"A2"*{\bullet}, 
"A3"*{\bullet}, 
"A4"*{\bullet}, 
"A5"*{\bullet}, 
(0,-0.5)*{+1}, 
(2.5,0.5)*{+1}, 
"B1"*{p_2}, 
"B2"*{p_1}, 
"B3"*{p_5}, 
"B4"*{p_4}, 
"B5"*{p_3}, 
(5.5,0)*{\Sigma^\circ:}, 
\endxy
\]

The above definitions are equivalent to the definitions in~\cite{dual,Dlam}. In~\cite{Dlam}, a $\mathcal{D}$-lamination was understood as a collection of closed loops on the surface $\Sigma_{\mathcal{D}}$ obtained by gluing $\Sigma$ and~$\Sigma^\circ$ along corresponding boundary edges.

\section{Duality map for the quantum Poisson variety}
\label{sec:DualityMapForTheQuantumPoissonVariety}

\subsection{Realization of the cluster Poisson variety}

We have associated a quantum cluster algebra $\mathcal{A}$ to the disk $\Sigma$. Let $\mathcal{F}$ be the ambient skew-field of this algebra $\mathcal{A}$.

\begin{definition}
Let $T$ be an ideal triangulation of $\Sigma$. For any $j\in J$, we define an element $X_j=X_{j;T}$ of $\mathcal{F}$ by the formula 
\[
X_j = M_T\big(\sum_s\varepsilon_{js}\mathbf{e}_s\big)
\]
where the $\mathbf{e}_s$ are standard basis vectors. In addition, we will use the notation $q=\omega^4$.
\end{definition}

As the notation suggests, these elements are related to the generators of the quantum Poisson variety.

\begin{proposition}
\label{prop:Xcommutation}
The elements $X_j$~($j\in J$) satisfy the commutation relations 
\[
X_iX_j=q^{2\varepsilon_{ij}}X_jX_i.
\]
\end{proposition}

\begin{proof}
By the properties of toric frames, we have 
\[
X_iX_j=\omega^{-2\Lambda_T(\sum_s\varepsilon_{is}\mathbf{e}_s,\sum_t\varepsilon_{jt}\mathbf{e}_t)} X_jX_i.
\]
Now the compatibility condition proved in Proposition~\ref{prop:compatibility} implies 
\begin{align*}
\Lambda_T\big(\sum_s\varepsilon_{is}\mathbf{e}_s,\sum_t\varepsilon_{jt}\mathbf{e}_t\big) &= \sum_{s,t}\varepsilon_{is}\varepsilon_{jt}\Lambda_T(\mathbf{e}_s,\mathbf{e}_t) \\
&= \sum_{s,t}\varepsilon_{is}\varepsilon_{jt}\lambda_{st} \\
&= -\sum_s\varepsilon_{is}\sum_tb_{tj}\lambda_{ts} \\
&= -4\varepsilon_{ij}
\end{align*}
where we have used the fact that $\varepsilon_{jt}=-b_{tj}$. Therefore $X_iX_j=\omega^{8\varepsilon_{ij}}X_jX_i=q^{2\varepsilon_{ij}}X_jX_i$.
\end{proof}

In the following result, $T'$ denotes the triangulation obtained from $T$ by a flip of the edge~$k$.

\begin{proposition}
\label{prop:Xtrans}
For each $i$, let $X_i'=X_{i;T'}$. Then 
\[
X_i'=
\begin{cases}
X_i\prod_{p=0}^{|\varepsilon_{ik}|-1}(1+q^{2p+1}X_k) & \mbox{if } \varepsilon_{ik}\leq0 \mbox{ and } i\neq k \\
X_iX_k^{\varepsilon_{ik}}\prod_{p=0}^{\varepsilon_{ik}-1}(X_k+q^{2p+1})^{-1} & \mbox{if } \varepsilon_{ik}\geq0 \mbox{ and } i\neq k \\
X_k^{-1} & \mbox{if } i=k.
\end{cases}
\]
\end{proposition}

\begin{proof}
According to Lemma~5.4 of~\cite{Tran}, a mutation in the direction $k$ transforms $X_i$ to the new element 
\[
X_i'=
\begin{cases}
X_i\prod_{p=0}^{|b_{ki}|-1}(1+\omega^{-2(-d_kp-\frac{d_k}{2})}X_k) & \mbox{if } b_{ki}\leq0 \mbox{ and } i\neq k \\
X_iX_k^{b_{ki}}\prod_{p=0}^{b_{ki}-1}(X_k+\omega^{-2(-d_kp-\frac{d_k}{2})})^{-1} & \mbox{if } b_{ki}\geq0 \mbox{ and } i\neq k \\
X_k^{-1} & \mbox{if } i=k
\end{cases}
\]
where $d_k$ is the number appearing in Definition~\ref{def:compatibility}. The statement follows if we substitute $d_k=4$ into this formula.
\end{proof}

\subsection{Construction of the map}

We will now define the map $\mathbb{I}_{\mathcal{A}}^q$. To do this, let $l$ be a point of $\mathcal{A}_\Sigma(\mathbb{Z}^t)$, represented by a collection of curves on~$\Sigma$. By the definition of the space $\mathcal{A}_\Sigma(\mathbb{Z}^t)$, these curves can be taken to be nonintersecting, so there exists an ideal triangulation $T_l$ of $\Sigma$ such that each curve coincides with an edge of $T_l$. Let 
\[
\mathbf{w}=(w_1,\dots,w_m)
\]
be the integral vector whose $i$th component $w_i$ is the weight of the curve corresponding to the edge $i$ of $T_l$ or zero if there is no such curve. Here $m=|I_{T_l}|$ is the number of edges of the ideal triangulation~$T_l$.

\begin{definition}
We will write $\mathbb{I}_{\mathcal{A}}^q(l)$ for the element of $\mathcal{F}$ given by 
\[
\mathbb{I}_{\mathcal{A}}^q(l)=M_{T_l}(\mathbf{w}).
\]
\end{definition}

It is clear from the definition of the toric frames that this is independent of the choice of~$T_l$. Our goal is to show that this definition provides a map $\mathbb{I}_{\mathcal{A}}^q:\mathcal{A}_\Sigma(\mathbb{Z}^t)\rightarrow\mathcal{X}_\Sigma^q$. To prove this, we will need a technical lemma on the $\mathbf{g}$-vectors from Section~\ref{sec:QuantumFPolynomials}.

\begin{lemma}
\label{lem:gsum}
Fix an ideal triangulation $T$ of $\Sigma=(\mathbb{D},\mathbb{M})$, and let $c_1,\dots,c_{2k}$ be noncontractible curves connecting marked points. Assume these curves form a closed path on~$\mathbb{D}$ in which $c_i$ and $c_{i+1}$ share a common endpoint for $i=1,\dots,2k$ where we number the indices modulo~$2k$. Let $\mathbf{g}_{c_i}$ be the $\mathbf{g}$-vector associated to $c_i$. Then the alternating sum 
\[
\mathbf{s}=\sum_{i=1}^{2k}(-1)^i\mathbf{g}_{c_i}
\]
is an integral linear combination of expressions of the form $\sum_s\varepsilon_{is}\mathbf{e}_s$.
\end{lemma}

\begin{proof}
Consider the triangle in $T$ that includes $p\in I-J$ as one of its edges and label the other edges as follows:
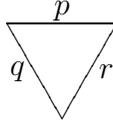
\[
\xy /l1pc/:
{\xypolygon3"A"{~:{(2,0):}}};
(2.4,0.5)*{q};
(-0.4,0.5)*{r}; 
(1,-1.4)*{p};
\endxy
\]
We can define a vector associated to this edge $p$ as $\mathbf{y}_p=\mathbf{e}_p+\mathbf{e}_q-\mathbf{e}_r$. Fix an edge $j\in J$. For any endpoint $v$ of $j$, there is a collection of edges of $T$ that start at $v$ and lie in the counterclockwise direction from~$j$. Consider a path $j'$ that goes diagonally across $j$, intersecting each of these edges transversally and terminating on the boundary. An example is illustrated below.
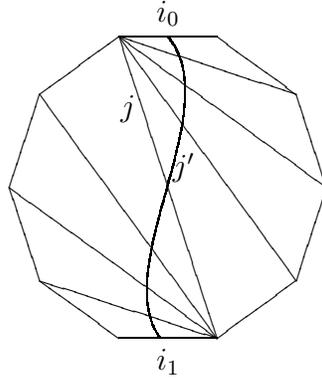
\[
\xy /l2.5pc/:
{\xypolygon10"A"{~:{(2,0):}~>{.}}};
{"A9"\PATH~={**@{.}}'"A4"},
{"A9"\PATH~={**@{.}}'"A5"},
{"A9"\PATH~={**@{.}}'"A6"},
{"A9"\PATH~={**@{.}}'"A7"},
{"A10"\PATH~={**@{.}}'"A4"},
{"A1"\PATH~={**@{.}}'"A4"},
{"A2"\PATH~={**@{.}}'"A4"},
"A1"*{\sbullet}, 
"A2"*{\sbullet}, 
"A3"*{\sbullet}, 
"A4"*{\sbullet}, 
"A5"*{\sbullet}, 
"A6"*{\sbullet}, 
"A7"*{\sbullet}, 
"A8"*{\sbullet}, 
"A9"*{\sbullet}, 
"A10"*{\sbullet}, 
(1,-1.9)*{}="1";
(1.1,1.9)*{}="2";
"1";"2" **\crv{(0.25,-1) & (1.75,1)};
(1.5,-1)*{j};
(0.8,-0.25)*{j'};
(1,-2.2)*{i_0};
(1,2.2)*{i_1};
\endxy
\]
Given such a path $j'$, let $E_j$ be the set of all edges in $T$ that $j'$ crosses. Then we can form the sum 
\[
\mathbf{p}_j=\mathbf{y}_{i_0}+\sum_{i\in E_j}\mathbf{x}_i+\mathbf{y}_{i_1}
\]
where $i_0$ and $i_1$ are the edges on which $j'$ terminates and we have written $\mathbf{x}_i=\sum_s\varepsilon_{is}\mathbf{e}_s$. One can check that this expression $\mathbf{p}_j$ equals $2\mathbf{e}_j$. Now let $c$ be a curve connecting marked points. It follows from the results of~\cite{MSW1} that the $\mathbf{g}$-vector $\mathbf{g}_c$ is an alternating sum of standard basis vectors corresponding to a path in~$T$. (A slightly technical point is that we work with \emph{extended} $\mathbf{g}$-vectors whereas \cite{MSW1} works with ordinary $\mathbf{g}$-vectors. To apply the results of~\cite{MSW1}, one should replace the surface $\Sigma$ by a larger surface $\Sigma'$ with twice as many marked points by gluing a triangle to each boundary segment of~$\Sigma$.) In fact, the entire sum appearing in the statement of the lemma can be written as an alternating sum $\mathbf{s}=\sum_{l=1}^{2N}(-1)^l\mathbf{e}_{j_l}$ of standard basis vectors associated to the edges $j_1,\dots,j_{2N}$ of a closed path in~$T$. We assume that these edges are ordered consecutively. To each edge $j$ of this path, we associate the arc $j'$ described above. We have 
\[
2\mathbf{s}=2\sum_{l=1}^{2N}(-1)^l\mathbf{e}_{j_l}=\sum_{l=1}^{2N}(-1)^l\mathbf{p}_{j_l}
\]
and all of the $\mathbf{y}$-terms cancel on the right hand side. Thus this expression is linear combination of the $\mathbf{x}_k$. By construction, these $\mathbf{x}_k$ are associated to the points of intersection between edges of~$T$ and a certain closed path on~$\Sigma$. Each edge intersects this path in an even number of points, so the vector $2\mathbf{s}$ is in fact a linear combination of the~$\mathbf{x}_k$ with \emph{even} coefficients. Dividing by~2 yields the desired result.
\end{proof}

\begin{theorem}
\label{thm:Xcanonicalregular}
Let $T$ be an ideal triangulation of $\Sigma$, and let $X_j=X_{j;T}$ be defined as above. Then for any $l\in\mathcal{A}_\Sigma(\mathbb{Z}^t)$, the element $\mathbb{I}_{\mathcal{A}}^q(l)$ is a Laurent polynomial in the~$X_j$ with coefficients in~$\mathbb{Z}_{\geq0}[q,q^{-1}]$.
\end{theorem}

\begin{proof}
By applying the relations from Definition~\ref{def:skein}, we can write 
\[
\mathbb{I}_{\mathcal{A}}^q(l) = \omega^\alpha\prod_{i=1}^mM_{T_l}(\mathbf{e}_i)^{w_i}
\]
where $\alpha=\sum_{i<j}\Lambda_{T_l}(\mathbf{e}_i,\mathbf{e}_j)w_iw_j$. Then Corollary~\ref{cor:Fpositivity} implies 
\[
\mathbb{I}_{\mathcal{A}}^q(l) = \omega^\alpha\prod_{i=1}^m(F_i\cdot M_T(\mathbf{g}_i))^{w_i}.
\]
Each $F_i$ in this last line denotes a polynomial in the expressions $M_T(\sum_j\varepsilon_{kj}\mathbf{e}_j)$ with coefficients in $\mathbb{Z}_{\geq0}[\omega^4,\omega^{-4}]$. The $\mathbf{g}_i$ are integral vectors. We have 
\begin{align*}
\Lambda_T\big(\sum_i\varepsilon_{ki}\mathbf{e}_i,\mathbf{e}_j\big) = \sum_i\varepsilon_{ki}\lambda_{ij} = \sum_i b_{ik}\lambda_{ij} = 4\delta_{kj}
\end{align*}
so that 
\[
M_T\big(\sum_i\varepsilon_{ki}\mathbf{e}_i\big)M_T(\mathbf{e}_j) = \omega^{-8\delta_{kj}} M_T(\mathbf{e}_j)M_T\big(\sum_i\varepsilon_{ki}\mathbf{e}_i\big).
\]
It follows that we can rearrange the factors in the last expression for $\mathbb{I}_{\mathcal{A}}^q(l)$ to get 
\[
\mathbb{I}_{\mathcal{A}}^q(l) = \omega^\alpha Q\cdot\prod_{i=1}^m M_T(\mathbf{g}_i)^{w_i}
\]
where $Q$ is a polynomial in the expressions $M_T(\sum_j\varepsilon_{kj}\mathbf{e}_j)$ with coefficients in the semiring $\mathbb{Z}_{\geq0}[\omega^4,\omega^{-4}]$. By Lemma~6.2 in~\cite{Tran}, we know that $\Lambda_{T_l}(\mathbf{e}_i,\mathbf{e}_j)=\Lambda_T(\mathbf{g}_i,\mathbf{g}_j)$, and so 
\begin{align*}
M_T\big(\sum_iw_i\mathbf{g}_i\big) &= \omega^{\sum_{i<j}\Lambda_T(\mathbf{g}_i,\mathbf{g}_j)w_iw_j} \prod_{i=1}^m M_T(\mathbf{g}_i)^{w_i} \\
&= \omega^\alpha \prod_{i=1}^m M_T(\mathbf{g}_i)^{w_i}.
\end{align*}
Hence 
\[
\mathbb{I}_{\mathcal{A}}^q(l) = Q\cdot M_T\big(\sum_iw_i\mathbf{g}_i\big).
\]
The vector $\sum_iw_i\mathbf{g}_i$ is a sum of vectors of the type appearing in the statement of Lemma~\ref{lem:gsum}, so this lemma implies that the second factor is a monomial in the $X_j$ with coefficient in $\mathbb{Z}_{\geq0}[\omega^4,\omega^{-4}]$. This completes the proof.
\end{proof}

By the properties established in Propositions~\ref{prop:Xcommutation}, and~\ref{prop:Xtrans}, we can regard the Laurent polynomial of Theorem~\ref{thm:Xcanonicalregular} as an element of the algebra $\mathcal{X}_\Sigma^q$. Thus we have constructed a canonical map $\mathbb{I}_{\mathcal{A}}^q:\mathcal{A}_\Sigma(\mathbb{Z}^t)\rightarrow\mathcal{X}_\Sigma^q$ as desired.

\subsection{Properties}

Fix an ideal triangulation $T$ of~$\Sigma$, and let $X_j$ be the elements of~$\mathcal{F}$ defined above. We will now prove several properties of the map $\mathbb{I}_{\mathcal{A}}^q$ conjectured in~\cite{IHES,ensembles}.

\begin{theorem}
\label{thm:Xproperties}
The map $\mathbb{I}_{\mathcal{A}}^q:\mathcal{A}_\Sigma(\mathbb{Z}^t)\rightarrow\mathcal{X}_\Sigma^q$ satisfies the following properties:
\begin{enumerate}
\item Each $\mathbb{I}_{\mathcal{A}}^q(l)$ is a Laurent polynomial in the variables $X_j$ with coefficients in $\mathbb{Z}_{\geq0}[q,q^{-1}]$.

\item The expression $\mathbb{I}_{\mathcal{A}}^q(l)$ agrees with the Laurent polynomial $\mathbb{I}_{\mathcal{A}}(l)$ when~$q=1$.

\item Let $*$ be the canonical involutive antiautomorphism of $\mathcal{X}^q$ that fixes each $X_i$ and sends~$q$ to~$q^{-1}$. Then $*\mathbb{I}_{\mathcal{A}}^q(l)=\mathbb{I}_{\mathcal{A}}^q(l)$.

\item The highest term of $\mathbb{I}_{\mathcal{A}}^q(l)$ is
\[
q^{-\sum_{i<j}\varepsilon_{ij}a_ia_j}X_1^{a_1}\dots X_n^{a_n}
\]
where $X_1^{a_1}\dots X_n^{a_n}$ is the highest term of the classical expression $\mathbb{I}_{\mathcal{A}}(l)$.

\item For any $l$,~$l'\in\mathcal{A}_\Sigma(\mathbb{Z}^t)$, we have
\[
\mathbb{I}_{\mathcal{A}}^q(l)\mathbb{I}_{\mathcal{A}}^q(l') = \sum_{l''\in\mathcal{A}_\Sigma(\mathbb{Z}^t)}c^q(l,l';l'')\mathbb{I}_{\mathcal{A}}^q(l'')
\]
where $c^q(l,l';l'')\in\mathbb{Z}[q,q^{-1}]$ and only finitely many terms are nonzero.
\end{enumerate}
\end{theorem}

\begin{proof}
1. This was proved in Theorem~\ref{thm:Xcanonicalregular}.

2. This follows immediately from the definition of $\mathbb{I}_{\mathcal{A}}$ given in~\cite{dual}.

3. If $K$ is any framed link in~$\Sigma$, let $K^\dagger$ be the framed link with the same underlying multicurve and the order of each crossing reversed. By Proposition~3.11 of~\cite{Muller} the map sending $[K]$ to $[K^\dagger]$ and $\omega$ to $\omega^\dagger\coloneqq\omega^{-1}$ extends to an involutive antiautomorphism of the skein algebra $\mathrm{Sk}_\omega(\Sigma)$. It extends further to an antiautomorphism of the fraction field $\mathcal{F}$ by the rule $(xy^{-1})^\dagger=(y^\dagger)^{-1}x^\dagger$. This operation is compatible with $*$ in the sense that 
\[
X_j^\dagger=X_j=*X_j
\]
and 
\[
q^\dagger=q^{-1}=*q.
\]
It is easy to check that $\dagger$ preserves $M_{T_l}(\mathbf{w})=\omega^{\sum_{i<j}\Lambda_{T_l}(\mathbf{e}_i,\mathbf{e}_j)}M_{T_l}(\mathbf{e}_1)^{w_1}\dots M_{T_l}(\mathbf{e}_n)^{w_m}$. It follows that $\mathbb{I}_{\mathcal{A}}^q(l)$ is $*$-invariant.

4. By part~1, we can write
\[
\mathbb{I}_{\mathcal{A}}^q(l) = \sum_{(i_1,\dots,i_n)\in\supp(l)}c_{i_1,\dots,i_n}X_1^{i_1}\dots X_n^{i_n}
\]
for some finite subset $\supp(l)\subseteq\mathbb{Z}^n$ where the coefficient $c_{i_1,\dots,i_n}\in\mathbb{Z}_{\geq0}[q,q^{-1}]$ is nonzero for all $(i_1,\dots,i_n)\in\supp(l)$. Consider the coefficient $c=c_{a_1,\dots,a_n}$. We can expand this coefficient as $c=\sum_sc_sq^s$. By setting $q=1$, we see that $\sum_sc_s=1$. It follows that we must have $c_s=1$ for some $s$, and all other terms must vanish. Thus $c=q^s$, and the leading term of $\mathbb{I}_{\mathcal{A}}^q(l)$ has the form
\[
q^sX_1^{a_1}\dots X_n^{a_n}
\]
for some integer $s$. By part~3, we know that the expression $\mathbb{I}_{\mathcal{A}}^q(l)$, and hence its leading term, is $*$-invariant. This means 
\[
q^{-s} X_n^{a_n}\dots X_1^{a_1}=q^s X_1^{a_1}\dots X_n^{a_n}.
\]
We have 
\begin{align*}
X_n^{a_n}\dots X_1^{a_1} &= q^{-2\sum_{j=2}^n\varepsilon_{1j}a_1a_j} X_1^{a_1}(X_n^{a_n}\dots X_2^{a_2}) \\
&= q^{-2\sum_{j=2}^n\varepsilon_{1j}a_1a_j} q^{-2\sum_{j=3}^n\varepsilon_{2j}a_2a_j} X_1^{a_1}X_2^{a_2} (X_n^{a_n}\dots X_3^{a_3}) \\
&= \dots \\
&= q^{-2\sum_{i<j}\varepsilon_{ij}a_ia_j} X_1^{a_1}\dots X_n^{a_n},
\end{align*}
so $*$-invariance of the leading term is equivalent to 
\[
q^{-s-2\sum_{i<j}\varepsilon_{ij}a_ia_j} X_1^{a_1}\dots X_n^{a_n}=q^s X_1^{a_1}\dots X_n^{a_n}.
\]
Equating coefficients, we see that $s=-\sum_{i<j}\varepsilon_{ij}a_ia_j$. Thus we see that the highest term equals $q^{-\sum_{i<j}\varepsilon_{ij}a_ia_j}X_1^{a_1}\dots X_n^{a_n}$.

5. Let $l$,~$l'\in\mathcal{A}_\Sigma(\mathbb{Z}^t)$. Then there exist ideal triangulations $T_l$ and $T_{l'}$ of~$\Sigma$ and integral vectors $\mathbf{w}$ and $\mathbf{w}'$ such that $\mathbb{I}_{\mathcal{A}}^q(l)=M_{T_l}(\mathbf{w})$ and $\mathbb{I}_{\mathcal{A}}^q(l')=M_{T_{l'}}(\mathbf{w}')$. For any integral vector $\mathbf{v}=(v_1,\dots,v_m)$, we can consider the vector $\mathbf{v}_+$ whose $i$th component equals $v_i$ if $v_i\geq0$ and zero otherwise. We can likewise consider the vector $\mathbf{v}_-$ whose $i$th component equals $v_i$ if $v_i\leq0$ and zero otherwise. Then there exists an integer $N$ such that 
\begin{align*}
\mathbb{I}_{\mathcal{A}}^q(l)\mathbb{I}_{\mathcal{A}}^q(l') &= M_{T_l}(\mathbf{w}) M_{T_{l'}}(\mathbf{w}') \\
&= \omega^N M_{T_l}(\mathbf{w}_-)M_{T_l}(\mathbf{w}_+)M_{T_{l'}}(\mathbf{w}_+')M_{T_{l'}}(\mathbf{w}_-') \\
&= \omega^N M_{T_l}(\mathbf{w}_-)[T_l^{\mathbf{w}_+}][T_{l'}^{\mathbf{w}_+'}]M_{T_{l'}}(\mathbf{w}_-').
\end{align*}
By applying the skein relations, we can write the product $[T_l^{\mathbf{w}_+}][T_{l'}^{\mathbf{w}_+'}]$ as 
\[
[T_l^{\mathbf{w}_+}][T_{l'}^{\mathbf{w}_+'}]=\sum_{i=1}^kd_i^\omega[K_i]
\]
where the $K_i$ are distinct simple multicurves on~$\Sigma$ and $d_i^\omega\in\mathbb{Z}[\omega,\omega^{-1}]$ are nonzero. Therefore 
\[
\mathbb{I}_{\mathcal{A}}^q(l)\mathbb{I}_{\mathcal{A}}^q(l') = \sum_{i=1}^k\omega^Nd_i^\omega M_{T_i}(\mathbf{w}_-)M_{T_i}(\mathbf{v}_i)M_{T_i}(\mathbf{w}_-')
\]
where $T_i$ is an ideal triangulation of $\Sigma$ such that each each curve of $K_i$ coincides with an edge of $T_i$. Here we have written $\mathbf{v}_i$ for the unique integral vector such that $[K_i]=M_{T_i}(\mathbf{v}_i)$. Each of the products $\omega^Nd_i^\omega M_{T_i}(\mathbf{w}_-)M_{T_i}(\mathbf{v}_i)M_{T_i}(\mathbf{w}_-')$ in this last expression can be written as $c^\omega(l,l';l_i)\mathbb{I}_\mathcal{A}^q(l_i)$ for some $l_i\in\mathcal{A}_\Sigma(\mathbb{Z}^t)$ and some $c^\omega(l,l';l_i)\in\mathbb{Z}[\omega,\omega^{-1}]$. Hence 
\[
\mathbb{I}_{\mathcal{A}}^q(l)\mathbb{I}_{\mathcal{A}}^q(l') = \sum_{i=1}^kc^\omega(l,l';l_i)\mathbb{I}_\mathcal{A}^q(l_i).
\]
The left hand side of this last equation is a Laurent polynomial with coefficients in~$\mathbb{Z}[q,q^{-1}]$. Let us write $[\mathbb{I}_{\mathcal{A}}^q(l_i)]^H$ for the highest term of $\mathbb{I}_{\mathcal{A}}^q(l_i)$. We can impose a lexicographic total ordering $\geq$ on the set of all commutative Laurent monomials in $X_1,\dots,X_n$ so that $[\mathbb{I}_{\mathcal{A}}^1(l_i)]^H$ is indeed the highest term of $\mathbb{I}_{\mathcal{A}}^1(l_i)$ with respect to this total ordering. These highest terms are distinct, so we may assume 
\[
[\mathbb{I}_{\mathcal{A}}^1(l_1)]^H > [\mathbb{I}_{\mathcal{A}}^1(l_2)]^H > \dots > [\mathbb{I}_{\mathcal{A}}^1(l_k)]^H.
\]
Consider the expression $c^\omega(l,l';l_1)[\mathbb{I}_{\mathcal{A}}^q(l_1)]^H$. It cannot cancel with any other term in the sum, so we must have $c^\omega(l,l';l_1)\in\mathbb{Z}[q,q^{-1}]$. It follows that the sum 
\[
\sum_{i=2}^kc^\omega(l,l';l_i)\mathbb{I}_\mathcal{A}^q(l_i)
\]
is a Laurent polynomial with coefficients in~$\mathbb{Z}[q,q^{-1}]$. Arguing as before, we see that $c^\omega(l,l';l_2)\in\mathbb{Z}[q,q^{-1}]$. Continuing in this way, we see that all $c^\omega(l,l';l_i)$ lie in $\mathbb{Z}[q,q^{-1}]$. This completes the proof.
\end{proof}

In \cite{IHES,ensembles}, Fock and Goncharov conjectured in addition that the Laurent polynomials $c^q(l,l',l'')$ have \emph{positive} coefficients. This property should be closely related to Conjecture~4.20 of~\cite{Thurston}. However, it does not obviously follow from the skein relations as these may involve negative coefficients.

\section{Duality map for the quantum symplectic double}
\label{sec:DualityMapForTheQuantumSymplecticDouble}

\subsection{Realization of the symplectic double}

We have associated a quantum cluster algebra $\mathcal{A}$ to the disk $\Sigma$. One can likewise associate a quantum cluster algebra $\mathcal{A}^\circ$ to the disk $\Sigma^\circ$ obtained from~$\Sigma$ by reversing its orientation. If $T$ is any ideal triangulation of $\Sigma$, then there is a corresponding triangulation $T^\circ$ of~$\Sigma^\circ$. We will write $\Lambda_{T^\circ}$ and $M_{T^\circ}$, respectively, for the compatibility matrix and toric frame corresponding to the triangulation~$T^\circ$. We will write $\mathbf{e}_i$ for the standard basis vector in $\mathbb{Z}^m$ associated to the edge $i$ of $T$ or the corresponding edge of $T^\circ$. Note that we have 
\[
\Lambda_{T^\circ}(\mathbf{e}_i,\mathbf{e}_j)=-\Lambda_T(\mathbf{e}_i,\mathbf{e}_j)
\]
for all $i$ and~$j$. Let $\mathcal{F}$ be the ambient skew-field of the quantum cluster algebra $\mathcal{A}$, and let $\mathcal{F}^\circ$ be the skew-field of the quantum cluster algebra $\mathcal{A}^\circ$. Each of these skew-fields is a two-sided module over the ring $\mathbb{Z}[\omega,\omega^{-1}]$, and we will consider the following elements of the tensor product $\mathcal{F}\otimes_{\mathbb{Z}[\omega,\omega^{-1}]}\mathcal{F}^\circ$.

\begin{definition}
\label{def:tensorgenerators}
Let $T$ be an ideal triangulation of $\Sigma$. For any $i\in I$ and any $j\in J$, we define elements $B_i=B_{i;T}$ and $X_j=X_{j;T}$ of $\mathcal{F}\otimes_{\mathbb{Z}[\omega,\omega^{-1}]}\mathcal{F}^\circ$ by the formulas 
\begin{align*}
X_j &= M_T\big(\sum_s\varepsilon_{js}\mathbf{e}_s\big)\otimes1, \\
B_i &= M_T(-\mathbf{e}_i)\otimes M_{T^\circ}(\mathbf{e}_i).
\end{align*}
In addition, we will use the notation $q=\omega^4$.
\end{definition}

Let us denote by $\mathcal{G}$ the subalgebra of $\mathcal{F}\otimes_{\mathbb{Z}[\omega,\omega^{-1}]}\mathcal{F}^\circ$ consisting of Laurent polynomials in the variables $B_i$ for $i\in I$ whose coefficients are rational expressions in the $X_j$ with coefficients in $\mathbb{Z}[\omega,\omega^{-1}]$. It is easy to check that the $B_i$ for $i\in I-J$ are central elements of this algebra, and in what follows it will be important to consider the quotient $\mathcal{F}_{\mathcal{D}}=\mathcal{G}/\mathcal{I}$ where $\mathcal{I}$ is the two-sided ideal of $\mathcal{G}$ generated by $B_i-1$ for $i\in I-J$.

The following result relates the elements $X_j$ and $B_j$ to the generators appearing in the definition of the quantum symplectic double.

\begin{proposition}
\label{prop:commutation}
The elements $B_j$ and $X_j$~($j\in J$) satisfy the following commutation relations:
\[
X_iX_j=q^{2\varepsilon_{ij}}X_jX_i, \quad B_iB_j=B_jB_i, \quad X_iB_j=q^{2\delta_{ij}}B_jX_i.
\]
\end{proposition}

\begin{proof}
The proof of the first relation is identical to the proof of Proposition~\ref{prop:Xcommutation}. To prove the second relation, observe that 
\begin{align*}
B_iB_j &= M_T(-\mathbf{e}_i)M_T(-\mathbf{e}_j)\otimes M_{T^\circ}(\mathbf{e}_i)M_{T^\circ}(\mathbf{e}_j) \\
&= \omega^{-2(\Lambda_T(\mathbf{e}_i,\mathbf{e}_j)+\Lambda_{T^\circ}(\mathbf{e}_i,\mathbf{e}_j))} M_T(-\mathbf{e}_j)M_T(-\mathbf{e}_i)\otimes M_{T^\circ}(\mathbf{e}_j)M_{T^\circ}(\mathbf{e}_i) \\
&= \omega^{-2(\Lambda_T(\mathbf{e}_i,\mathbf{e}_j)+\Lambda_{T^\circ}(\mathbf{e}_i,\mathbf{e}_j))} B_jB_i
\end{align*}
and 
\[
\Lambda_T(\mathbf{e}_i,\mathbf{e}_j)+\Lambda_{T^\circ}(\mathbf{e}_i,\mathbf{e}_j)=0.
\]
Therefore $B_iB_j=B_jB_i$. Finally, we have 
\begin{align*}
X_iB_j &= M_T\big(\sum_s\varepsilon_{is}\mathbf{e}_s\big)M_T(-\mathbf{e}_j)\otimes M_{T^\circ}(\mathbf{e}_j) \\
&= \omega^{-2\Lambda_T(\sum_s\varepsilon_{is}\mathbf{e}_s,-\mathbf{e}_j)} M_T(-\mathbf{e}_j)M_T\big(\sum_s\varepsilon_{is}\mathbf{e}_s\big)\otimes M_{T^\circ}(\mathbf{e}_j) \\
&= \omega^{-2\Lambda_T(\sum_s\varepsilon_{is}\mathbf{e}_s,-\mathbf{e}_j)} B_jX_i.
\end{align*}
By the compatibility condition, we have 
\begin{align*}
\Lambda_T\big(\sum_s\varepsilon_{is}\mathbf{e}_s,-\mathbf{e}_j\big) &= -\sum_s\varepsilon_{is}\Lambda_T(\mathbf{e}_s,\mathbf{e}_j) \\
&= -\sum_sb_{si}\lambda_{sj} \\
&= -4\delta_{ij}.
\end{align*}
Therefore $X_iB_j=\omega^{8\delta_{ij}}B_jX_i=q^{2\delta_{ij}}B_jX_i$.
\end{proof}

In the following result, $T'$ denotes the triangulation obtained from $T$ by a flip of the edge~$k$. We use the notation $\mathbb{B}_k^\pm$ defined in Section~\ref{sec:TheQuantumDoubleConstruction}.

\begin{proposition}
\label{prop:Btrans}
For each $i$, let $B_i'=B_{i;T'}$. Then 
\[
B_i'=
\begin{cases}
(qX_k\mathbb{B}_k^++\mathbb{B}_k^-)B_k^{-1}(1+q^{-1}X_k)^{-1} & \mbox{if } i=k \\
B_i & \mbox{if } i\neq k.
\end{cases}
\]
\end{proposition}

\begin{proof}
By Proposition~\ref{prop:exchangetoricframe}, we have 
\begin{align*}
M_{T'}(\mathbf{e}_k) &= M_T\big(-\mathbf{e}_k+\sum_i[\varepsilon_{ki}]_+\mathbf{e}_i\big) + M_T\big(-\mathbf{e}_k+\sum_i[-\varepsilon_{ki}]_+\mathbf{e}_i\big) \\
&= M_T\big(\sum_i\varepsilon_{ki}\mathbf{e}_i-\mathbf{e}_k+\sum_i[-\varepsilon_{ki}]_+\mathbf{e}_i\big) + M_T\big(-\mathbf{e}_k+\sum_i[-\varepsilon_{ki}]_+\mathbf{e}_i\big) \\
&= \omega^\alpha M_T\big(\sum_i\varepsilon_{ki}\mathbf{e}_i\big) M_T(-\mathbf{e}_k) M_T\big(\sum_i[-\varepsilon_{ki}]_+\mathbf{e}_i\big) \\
&\qquad\qquad + \omega^\beta M_T(-\mathbf{e}_k) M_T\big(\sum_i[-\varepsilon_{ki}]_+\mathbf{e}_i\big)
\end{align*}
where we have written 
\[
\alpha = \Lambda_T\big(\sum_i\varepsilon_{ki}\mathbf{e}_i-\mathbf{e}_k, \sum_i[-\varepsilon_{ki}]_+\mathbf{e}_i\big) + \Lambda_T\big(\sum_i\varepsilon_{ki}\mathbf{e}_i,-\mathbf{e}_k\big)
\]
and 
\[
\beta = \Lambda_T\big(-\mathbf{e}_k, \sum_i[-\varepsilon_{ki}]_+\mathbf{e}_i\big).
\]
Factoring, we obtain 
\[
M_{T'}(\mathbf{e}_k) = \left(1+\omega^{\alpha-\beta}M_T\big(\sum_i\varepsilon_{ki}\mathbf{e}_i\big)\right) \omega^\beta M_T(-\mathbf{e}_k) M_T\big(\sum_i[-\varepsilon_{ki}]_+\mathbf{e}_i\big).
\]
A straightforward calculation using the compatibility condition shows $\alpha-\beta=-4$. Substituting this into the last expression, we see that 
\[
M_{T'}(\mathbf{e}_k)\otimes1 = (1+q^{-1}X_k)\omega^\beta \left(M_T(-\mathbf{e}_k)\otimes1\right) \left(M_T\big(\sum_i[-\varepsilon_{ki}]_+\mathbf{e}_i\big)\otimes1\right).
\]
On the other hand, we have 
\begin{align*}
1\otimes M_{(T')^\circ}(\mathbf{e}_k) &= 1\otimes M_{T^\circ}\big(-\mathbf{e}_k+\sum_i[\varepsilon_{ki}]_+\mathbf{e}_i\big) + 1\otimes M_{T^\circ}\big(-\mathbf{e}_k+\sum_i[-\varepsilon_{ki}]_+\mathbf{e}_i\big) \\
&= M_T\big(\sum_i\varepsilon_{ki}\mathbf{e}_i - \sum_i[\varepsilon_{ki}]_+\mathbf{e}_i + \sum_i[-\varepsilon_{ki}]_+\mathbf{e}_i\big)\otimes M_{T^\circ}\big(-\mathbf{e}_k+\sum_i[\varepsilon_{ki}]_+\mathbf{e}_i\big) \\
&\qquad+ M_T\big(-\sum_i[-\varepsilon_{ki}]_+\mathbf{e}_i + \sum_i[-\varepsilon_{ki}]_+\mathbf{e}_i\big) \otimes M_{T^\circ}\big(-\mathbf{e}_k+\sum_i[-\varepsilon_{ki}]_+\mathbf{e}_i\big).
\end{align*}
By the properties of toric frames, this equals 
\begin{align*}
&\omega^\gamma M_T\big(\sum_i\varepsilon_{ki}\mathbf{e}_i\big) M_T\big( - \sum_i[\varepsilon_{ki}]_+\mathbf{e}_i\big) M_T\big(\sum_i[-\varepsilon_{ki}]_+\mathbf{e}_i\big) \otimes M_{T^\circ}\big(\sum_i[\varepsilon_{ki}]_+\mathbf{e}_i\big)M_{T^\circ}(-\mathbf{e}_k) \\
&\qquad+\omega^\delta M_T\big(-\sum_i[-\varepsilon_{ki}]_+\mathbf{e}_i\big) M_T\big(\sum_i[-\varepsilon_{ki}]_+\mathbf{e}_i\big) \otimes M_{T^\circ}\big(\sum_i[-\varepsilon_{ki}]_+\mathbf{e}_i\big) M_{T^\circ}(-\mathbf{e}_k)
\end{align*}
for certain exponents $\gamma$ and $\delta$. Factoring, we obtain 
\begin{align*}
&\biggr(\omega^{\gamma-\delta} M_T\big(\sum_i\varepsilon_{ki}\mathbf{e}_i\big) M_T\big( - \sum_i[\varepsilon_{ki}]_+\mathbf{e}_i\big) \otimes M_{T^\circ}\big(\sum_i[\varepsilon_{ki}]_+\mathbf{e}_i\big) \\
&\qquad+ M_T\big(-\sum_i[-\varepsilon_{ki}]_+\mathbf{e}_i\big) \otimes M_{T^\circ}\big(\sum_i[-\varepsilon_{ki}]_+\mathbf{e}_i\big)\biggr) \omega^\delta M_T\big(\sum_i[-\varepsilon_{ki}]_+\mathbf{e}_i\big) \otimes M_{T^\circ}(-\mathbf{e}_k) \\
&=(\omega^{\gamma-\delta}X_k\mathbb{B}_k^++\mathbb{B}_k^-) \omega^\delta \left(1\otimes M_{T^\circ}(-\mathbf{e}_k)\right) \biggr(M_T\big(\sum_i[-\varepsilon_{ki}]_+\mathbf{e}_i\big)\otimes1\biggr).
\end{align*}
A straightforward calculation shows that $\delta=\beta$ and $\gamma-\delta=4$. Substituting these into the last expression, we see that 
\[
1\otimes M_{(T')^\circ}(\mathbf{e}_k) = (qX_k\mathbb{B}_k^++\mathbb{B}_k^-) \omega^\beta \left(1\otimes M_{T^\circ}(-\mathbf{e}_k)\right) \biggr(M_T\big(\sum_i[-\varepsilon_{ki}]_+\mathbf{e}_i\big)\otimes1\biggr).
\]
Finally, by the definition of $B_k'$, we have 
\begin{align*}
B_k' &= M_{T'}(-\mathbf{e}_k)\otimes M_{(T')^\circ}(\mathbf{e}_k) =(1\otimes M_{(T')^\circ}(\mathbf{e}_k)) (M_{T'}(\mathbf{e}_k)\otimes1)^{-1} \\
&= (qX_k\mathbb{B}_k^++\mathbb{B}_k^-) \omega^\beta \left(1\otimes M_{T^\circ}(-\mathbf{e}_k)\right) \biggr(M_T\big(\sum_i[-\varepsilon_{ki}]_+\mathbf{e}_i\big)\otimes1\biggr) \\
&\qquad\cdot \biggr(M_T\big(\sum_i[-\varepsilon_{ki}]_+\mathbf{e}_i\big)\otimes1\biggr)^{-1} \left(M_T(-\mathbf{e}_k)\otimes1\right)^{-1} \omega^{-\beta} (1+q^{-1}X_k)^{-1} \\
&= (qX_k\mathbb{B}_k^++\mathbb{B}_k^-)B_k^{-1}(1+q^{-1}X_k)^{-1}.
\end{align*}
For $i\neq k$, we clearly have $B_i'=B_i$. This completes the proof.
\end{proof}

By Proposition~\ref{prop:Xtrans}, the transformation of the $X_j$ under a flip of the triangulation is given by the same rule derived in Appendix~\ref{app:DerivationOfMutationFormulas}.

\subsection{Construction of the map}

We are now ready to define the map $\mathbb{I}_{\mathcal{D}}^q$. To do this, let $l$ be a point of $\mathcal{D}_\Sigma(\mathbb{Z}^t)$, represented by a collection of curves on $\Sigma$ and~$\Sigma^\circ$. Write $\mathcal{C}$ for the collection of curves on~$\Sigma$ and $\mathcal{C}^\circ$ for the collection of curves on~$\Sigma^\circ$. There is an ideal triangulation $T_{\mathcal{C}}$ of $\Sigma$ such that each curve in $\mathcal{C}$ coincides with an edge of~$T_{\mathcal{C}}$, and there is an ideal triangulation $T_{\mathcal{C}^\circ}$ of~$\Sigma^\circ$ such that each curve in $\mathcal{C}^\circ$ coincides with an edge of $T_{\mathcal{C}^\circ}$. Fix an ideal triangulation $T$ of~$\Sigma$. Then each $c\in\mathcal{C}$ determines an integral $\mathbf{g}$-vector~$\mathbf{g}_c=\mathbf{g}_{c,T_{\mathcal{C}}}^T$. The triangulation $T$ determines a corresponding ideal triangulation of~$\Sigma^\circ$, and so we can likewise associate to each $c^\circ\in\mathcal{C}^\circ$ an integral vector $\mathbf{g}_{c^\circ}=\mathbf{g}_{c^\circ,T_{\mathcal{C}^\circ}}^{T}$.

\begin{lemma}
The number 
\[
N_l\coloneqq\Lambda_T(\mathbf{g}_{\mathcal{C}},\mathbf{g}_{\mathcal{C}^\circ}),
\]
where 
\[
\mathbf{g}_{\mathcal{C}}=\sum_{c\in\mathcal{C}}\mathbf{g}_c, \quad \mathbf{g}_{\mathcal{C}^\circ}=\sum_{c^\circ\in\mathcal{C}^\circ}\mathbf{g}_{c^\circ},
\]
is independent of the triangulation~$T$.
\end{lemma}

\begin{proof}
Let $T$ and $T'$ be two different triangulations. Let $\mathbf{g}_{c,T_{\mathcal{C}}}^T$ and $\mathbf{g}_{c^\circ,T_{\mathcal{C}^\circ}}^T$ be the $\mathbf{g}$-vectors associated to the arcs $c\in\mathcal{C}$ and $c^\circ\in\mathcal{C}^\circ$, respectively, defined using the triangulation $T$. Similarly, let $\mathbf{g}_{c,T_{\mathcal{C}}}^{T'}$ and $\mathbf{g}_{c^\circ,T_{\mathcal{C}^\circ}}^{T'}$ be the vectors defined using the triangulation $T'$. Number the edges of $T$ so that each edge is identified with some number $1,\dots,m$. Using the recurrence relations discussed in Section~3 of~\cite{Tran}, we can write 
\[
\mathbf{g}_{c,T_{\mathcal{C}}}^T = P_{c,T_{\mathcal{C}}}^T(\mathbf{e}_1,\dots,\mathbf{e}_m), \quad \mathbf{g}_{c^\circ,T_{\mathcal{C}^\circ}}^T = P_{c^\circ,T_{\mathcal{C}^\circ}}^T(\mathbf{e}_1,\dots,\mathbf{e}_m)
\]
for linear polynomials $P_{c,T_{\mathcal{C}}}^T$ and $P_{c^\circ,T_{\mathcal{C}^\circ}}^T$, and we also have 
\[
\mathbf{g}_{c,T_{\mathcal{C}}}^{T'} = P_{c,T_{\mathcal{C}}}^T(\mathbf{g}_{1,T}^{T'},\dots,\mathbf{g}_{m,T}^{T'}), \quad \mathbf{g}_{c^\circ,T_{\mathcal{C}^\circ}}^{T'} = P_{c^\circ,T_{\mathcal{C}^\circ}}^T(\mathbf{g}_{1,T}^{T'},\dots,\mathbf{g}_{m,T}^{T'})
\]
where $\mathbf{g}_{i,T}^{T'}$ is the $\mathbf{g}$-vector associated to the $i$th edge of $T$ and the triangulation $T'$. There exist coefficients $c_{p,q}\in\mathbb{Z}$ such that 
\[
P_{c,T_{\mathcal{C}}}^T(x_1,\dots,x_m)P_{c^\circ,T_{\mathcal{C}^\circ}}^T(x_1,\dots,x_m)=\sum_{p,q=1}^mc_{p,q}x_px_q\in\mathbb{Z}[x_1,\dots,x_m].
\]
It follows from Lemma~6.2 in~\cite{Tran} that 
\begin{align*}
\Lambda_{T'}(\mathbf{g}_{c,T_{\mathcal{C}}}^{T'},\mathbf{g}_{c^\circ,T_{\mathcal{C}^\circ}}^{T'}) &= \Lambda_{T'}\left(P_{c,T_{\mathcal{C}}}^T(\mathbf{g}_{1,T}^{T'},\dots,\mathbf{g}_{m,T}^{T'}),P_{c^\circ,T_{\mathcal{C}^\circ}}^T(\mathbf{g}_{1,T}^{T'},\dots,\mathbf{g}_{m,T}^{T'})\right) \\
&= \sum_{p,q=1}^mc_{p,q}\Lambda_{T'}(\mathbf{g}_{p,T}^{T'},\mathbf{g}_{q,T}^{T'}) \\
&= \sum_{p,q=1}^mc_{p,q}\Lambda_T(\mathbf{e}_p,\mathbf{e}_q) \\
&= \Lambda_T\left(P_{c,T_{\mathcal{C}}}^T(\mathbf{e}_1,\dots,\mathbf{e}_m),P_{c^\circ,T_{\mathcal{C}^\circ}}^T(\mathbf{e}_1,\dots,\mathbf{e}_m)\right) \\
&= \Lambda_T(\mathbf{g}_{c,T_{\mathcal{C}}}^T,\mathbf{g}_{c^\circ,T_{\mathcal{C}^\circ}}^T).
\end{align*}
Let us write $\mathbf{g}_{\mathcal{C}}^T=\sum_{c\in\mathcal{C}}\mathbf{g}_{c,T_{\mathcal{C}}}^T$ and $\mathbf{g}_{\mathcal{C}^\circ}^T=\sum_{c^\circ\in\mathcal{C}^\circ}\mathbf{g}_{c^\circ,T_{\mathcal{C}^\circ}}^T$ and also $\mathbf{g}_{\mathcal{C}}^{T'}=\sum_{c\in\mathcal{C}}\mathbf{g}_{c,T_{\mathcal{C}}}^{T'}$ and $\mathbf{g}_{\mathcal{C}^\circ}^{T'}=\sum_{c^\circ\in\mathcal{C}^\circ}\mathbf{g}_{c^\circ,T_{\mathcal{C}^\circ}}^{T'}$. Then 
\begin{align*}
\Lambda_{T'}(\mathbf{g}_{\mathcal{C}}^{T'},\mathbf{g}_{\mathcal{C}^\circ}^{T'}) &= \sum_{c\in\mathcal{C},c^\circ\in\mathcal{C}^\circ}\Lambda_{T'}(\mathbf{g}_{c,T_{\mathcal{C}}}^{T'},\mathbf{g}_{c^\circ,T_{\mathcal{C}^\circ}}^{T'}) \\
&= \sum_{c\in\mathcal{C},c^\circ\in\mathcal{C}^\circ}\Lambda_T(\mathbf{g}_{c,T_{\mathcal{C}}}^T,\mathbf{g}_{c^\circ,T_{\mathcal{C}^\circ}}^T) \\
&= \Lambda_T(\mathbf{g}_{\mathcal{C}}^T,\mathbf{g}_{\mathcal{C}^\circ}^T)
\end{align*}
as desired.
\end{proof}

We now come to the main definition of the present paper.

\begin{definition}
We will write $\mathbb{I}_{\mathcal{D}}^q(l)$ for the element of $\mathcal{F}\otimes_{\mathbb{Z}[\omega,\omega^{-1}]}\mathcal{F}^\circ$ given by 
\[
\mathbb{I}_{\mathcal{D}}^q(l)=\omega^{-N_l}\cdot[\mathcal{C}]^{-1}\otimes[\mathcal{C}^\circ].
\]
Here we are regarding $\mathcal{C}$ and $\mathcal{C}^\circ$ as simple multicurves on $\Sigma$ and $\Sigma^\circ$, respectively, and we are writing $[\mathcal{C}]$ and $[\mathcal{C}^\circ]$ for the corresponding classes in the skein algebra.
\end{definition}

Our goal is to show that the above definition provides a map $\mathbb{I}_{\mathcal{D}}^q:\mathcal{D}_\Sigma(\mathbb{Z}^t)\rightarrow\mathcal{D}_\Sigma^q$. Let 
\begin{align*}
\mathbf{w} &= (w_1,\dots,w_m), \\
\mathbf{w}^\circ &= (w_1^\circ,\dots,w_m^\circ)
\end{align*}
be integral vectors where $w_i$ is the weight of the curve corresponding to the edge $i$ of~$T_{\mathcal{C}}$ and $w_i^\circ$ is the weight of the curve corresponding to the edge $i$ of $T_{\mathcal{C}^\circ}$.

\begin{theorem}
\label{thm:Dcanonicalrational}
Let $T$ be an ideal triangulation of $\Sigma$, and let $B_i$,~$X_j$ for $i\in I$ and $j\in J$ be defined as above. Then for any $l\in\mathcal{D}_\Sigma(\mathbb{Z}^t)$, the element $\mathbb{I}_{\mathcal{D}}^q(l)$ is a Laurent polynomial in the~$B_i$ whose coefficients are rational expressions in the $X_j$ with coefficients in~$\mathbb{Z}_{\geq0}[q,q^{-1}]$.
\end{theorem}

\begin{proof}
By applying the relations from Definition~\ref{def:skein}, we can write 
\[
\mathbb{I}_{\mathcal{D}}^q(l) = \omega^{-N_l}\cdot\omega^\alpha\prod_{i=1}^mM_{T_{\mathcal{C}}}(\mathbf{e}_i)^{-w_i} \otimes\omega^{\alpha^\circ}\prod_{i=1}^mM_{T_{\mathcal{C}^\circ}}(\mathbf{e}_i)^{w_i^\circ}
\]
where
\[
\alpha=\sum_{i<j}\Lambda_{T_{\mathcal{C}}}(\mathbf{e}_i,\mathbf{e}_j)w_iw_j
\]
and 
\[
\alpha^\circ=\sum_{i<j}\Lambda_{T_{\mathcal{C}^\circ}}(\mathbf{e}_i,\mathbf{e}_j)w_i^\circ w_j^\circ.
\]
Then Corollary~\ref{cor:Fpositivity} implies 
\[
\mathbb{I}_{\mathcal{D}}^q(l) = \omega^{-N_l+\alpha+\alpha^\circ}\cdot\prod_{i=1}^m\left(F_i\cdot M_T(\mathbf{g}_i)\right)^{-w_i}\otimes \prod_{i=1}^m\left(F_i^\circ\cdot M_{T^\circ}(\mathbf{g}_i^\circ)\right)^{w_i^\circ}.
\]
Each $F_i$ in the last line denotes a polynomial in the expressions $M_T(\sum_j\varepsilon_{kj}\mathbf{e}_j)$ with coefficients in $\mathbb{Z}_{\geq0}[\omega^4,\omega^{-4}]$, and each $F_i^\circ$ denotes a polynomial in the expressions $M_{T^\circ}(\sum_j\varepsilon_{kj}\mathbf{e}_j)$ with coefficients in $\mathbb{Z}_{\geq0}[\omega^4,\omega^{-4}]$. The $\mathbf{g}_i$ and $\mathbf{g}_i^\circ$ are integral vectors. We have 
\begin{align*}
\Lambda_T\big(\sum_i\varepsilon_{ki}\mathbf{e}_i,\mathbf{e}_j\big) = \sum_i\varepsilon_{ki}\lambda_{ij} = \sum_i b_{ik}\lambda_{ij} = 4\delta_{kj}
\end{align*}
so that 
\[
M_T\big(\sum_i\varepsilon_{ki}\mathbf{e}_i\big)M_T(\mathbf{e}_j) = \omega^{-8\delta_{kj}} M_T(\mathbf{e}_j)M_T\big(\sum_i\varepsilon_{ki}\mathbf{e}_i\big).
\]
Similarly, we have 
\[
M_{T^\circ}\big(\sum_i\varepsilon_{ki}\mathbf{e}_i\big)M_{T^\circ}(\mathbf{e}_j) = \omega^{8\delta_{kj}} M_{T^\circ}(\mathbf{e}_j)M_{T^\circ}\big(\sum_i\varepsilon_{ki}\mathbf{e}_i\big).
\]
It follows that we can rearrange the factors in the last expression for $\mathbb{I}_{\mathcal{D}}^q(l)$ to get 
\begin{align*}
\mathbb{I}_{\mathcal{D}}^q(l) &= \omega^{-N_l+\alpha+\alpha^\circ} P\cdot\prod_{i=1}^m M_T(\mathbf{g}_i)^{-w_i}\otimes Q\cdot\prod_{i=1}^m M_{T^\circ}(\mathbf{g}_i^\circ)^{w_i^\circ} \\
&= \omega^{-N_l+\alpha+\alpha^\circ}(P\otimes Q)\cdot\big(\prod_{i=1}^m M_T(\mathbf{g}_i)^{-w_i}\otimes \prod_{i=1}^m M_{T^\circ}(\mathbf{g}_i^\circ)^{w_i^\circ}\big)
\end{align*}
where $P$ is a rational function in the expressions $M_T(\sum_i\varepsilon_{kj}\mathbf{e}_j)$ with coefficients in $\mathbb{Z}_{\geq0}[\omega^4,\omega^{-4}]$ and $Q$ is a polynomial in the expressions $M_{T^\circ}(\sum_i\varepsilon_{kj}\mathbf{e}_j)$ with coefficients in $\mathbb{Z}_{\geq0}[\omega^4,\omega^{-4}]$. We claim that the factor $P\otimes Q$ is a Laurent polynomial in the $B_j$ whose coefficients are rational expressions in the~$X_j$ with coefficients in~$\mathbb{Z}_{\geq0}[q,q^{-1}]$. Indeed, $P\otimes 1$ is a rational function in the $X_j$, while $1\otimes Q$ is a polynomial in the expressions 
\begin{align*}
1\otimes M_{T^\circ}\big(\sum_i\varepsilon_{ki}\mathbf{e}_i\big) &= M_T\big(\sum_i\varepsilon_{ki}\mathbf{e}_i-\sum_i\varepsilon_{ki}\mathbf{e}_i\big) \otimes M_{T^\circ}\big(\sum_i\varepsilon_{ki}\mathbf{e}_i\big) \\
&= M_T\big(\sum_i\varepsilon_{ki}\mathbf{e}_i\big)M_T\big(-\sum_i\varepsilon_{ki}\mathbf{e}_i\big) \otimes M_{T^\circ}\big(\sum_i\varepsilon_{ki}\mathbf{e}_i\big) \\
&= X_k\mathbb{B}_k.
\end{align*}
Hence the product 
\[
P\otimes Q=(P\otimes1)\cdot(1\otimes Q)
\]
has the desired form. By Lemma~6.2 in~\cite{Tran}, we have $\Lambda_{T_\mathcal{C}}(\mathbf{e}_i,\mathbf{e}_j)=\Lambda_T(\mathbf{g}_i,\mathbf{g}_j)$, and so we can write 
\[
\mathbb{I}_{\mathcal{D}}^q(l) = \omega^{-N_l}(P\otimes Q)\cdot\left(M_T(-\mathbf{g}_{\mathcal{C}})\otimes M_{T^\circ}(\mathbf{g}_{\mathcal{C}^\circ})\right).
\]
We have 
\[
M_T(-\mathbf{g}_{\mathcal{C}^\circ})\otimes M_{T^\circ}(\mathbf{g}_{\mathcal{C}^\circ})=\prod_{i=1}^mB_i^{g_i}
\]
for some integral exponents $g_i$, and so 
\begin{align*}
\mathbb{I}_{\mathcal{D}}^q(l) &= \omega^{-N_l}(P\otimes Q)\cdot\left(M_T(-\mathbf{g}_{\mathcal{C}}) M_T(\mathbf{g}_{\mathcal{C}^\circ})\otimes1\right)\cdot\prod_{i=1}^mB_i^{g_i} \\
&=(P\otimes Q)\cdot\left(M_T(-\mathbf{g}_{\mathcal{C}} + \mathbf{g}_{\mathcal{C}^\circ})\otimes1\right)\cdot\prod_{i=1}^mB_i^{g_i}.
\end{align*}
It follows from Lemma~\ref{lem:gsum} that the factor $M_T(-\mathbf{g}_{\mathcal{C}} + \mathbf{g}_{\mathcal{C}^\circ})\otimes1$ is a monomial in the $X_j$ with coefficients in $\mathbb{Z}_{\geq0}[\omega^4,\omega^{-4}]$. This completes the proof.
\end{proof}

Thus we see that $\mathbb{I}_{\mathcal{D}}^q(l)$ is an element of the algebra $\mathcal{G}$ introduced following Definition~\ref{def:tensorgenerators}. By the properties established in Propositions~\ref{prop:commutation}, \ref{prop:Btrans}, and~\ref{prop:Xtrans}, we can regard its image in the quotient $\mathcal{F}_{\mathcal{D}}$ as an element of the algebra $\mathcal{D}_\Sigma^q$. Thus we have constructed a canonical map $\mathbb{I}_{\mathcal{D}}^q:\mathcal{D}_\Sigma(\mathbb{Z}^t)\rightarrow\mathcal{D}_\Sigma^q$ as desired.

We should emphasize here that a function of the form $\mathbb{I}_{\mathcal{D}}^q(l)$ for $q=1$ need not be regular on all of~$\mathcal{D}_\Sigma$. Indeed, after choosing an ideal triangulation~$T$ of~$\Sigma$, we can define $l$ to consist of a single internal edge~$j$ of~$T$ on~$\Sigma$, together with the image of this edge on~$\Sigma^\circ$. In this case, we have $\mathbb{I}^1(l)=B_j$, and we have seen in Section~\ref{sec:TheClassicalLimit} that the expression $\mu_k^*(B_k)$ need not be a Laurent polynomial.

\subsection{Properties}

We will now examine the properties of our construction. To formulate our results precisely, we first note that there is a natural inclusion $\varphi:\mathcal{A}_\Sigma(\mathbb{Z}^t)\hookrightarrow\mathcal{D}_\Sigma(\mathbb{Z}^t)$. 

\begin{definition}
Let $l$ be a point of $\mathcal{A}_\Sigma(\mathbb{Z}^t)$ represented by a collection of curves on~$\Sigma$. Draw those curves having positive weight on the disk~$\Sigma^\circ$. The point $l$ determines a collection of curves of $\Sigma$ having negative weight. Take absolute values to get a collection of positive weights for these curves on~$\Sigma$. These data determine the point $\varphi(l)\in\mathcal{D}_\Sigma(\mathbb{Z}^t)$.
\end{definition}

The diagrams drawn in Section~\ref{sec:Laminations} show an example of a point $l\in\mathcal{A}_\Sigma(\mathbb{Z}^t)$ and the corresponding point $\varphi(l)\in\mathcal{D}_\Sigma(\mathbb{Z}^t)$. We also have a natural map $\pi^q:\mathcal{X}_\Sigma^q\rightarrow\mathcal{D}_\Sigma^q$.

\begin{definition}
The map $\pi^q$ is an algebra homomorphism defined on generators by $\pi^q(X_i)=X_i\mathbb{B}_i$.
\end{definition}

\begin{theorem}
The map $\mathbb{I}_{\mathcal{D}}^q:\mathcal{D}_\Sigma(\mathbb{Z}^t)\rightarrow\mathcal{D}_\Sigma^q$ satisfies the following properties:
\begin{enumerate}
\item Each $\mathbb{I}_{\mathcal{D}}^q(l)$ is a rational function in the variables $B_i$ and~$X_i$ with coefficients in $\mathbb{Z}_{\geq0}[q,q^{-1}]$.

\item The expression $\mathbb{I}_{\mathcal{D}}^q(l)$ agrees with the rational function $\mathbb{I}_{\mathcal{D}}(l)$ when $q=1$.

\item The following diagram commutes:
\[
\xymatrix{ 
\mathcal{A}_\Sigma(\mathbb{Z}^t) \ar[r]^-{\mathbb{I}_{\mathcal{A}}^q} \ar[d]_{\varphi} & \mathcal{X}_\Sigma^q \ar[d]^{\pi^q} \\
\mathcal{D}_\Sigma(\mathbb{Z}^t) \ar_-{\mathbb{I}_{\mathcal{D}}^q}[r] & \mathcal{D}_\Sigma^q.
}
\]
\end{enumerate}
\end{theorem}

\begin{proof}
1. This was proved in Theorem~\ref{thm:Dcanonicalrational}.

2. This follows immediately from the definition of $\mathbb{I}_{\mathcal{D}}$ given in~\cite{Dlam}.

3. Let $l$ be a point of $\mathcal{A}_\Sigma(\mathbb{Z}^t)$, represented by a collection of curves on~$\Sigma$. There exists an ideal triangulation $T_l$ of $\Sigma$ such that each of these curves coincides with an edge of~$T_l$. Let 
\[
\mathbf{w}=(w_1,\dots,w_m)
\]
be the integral vector whose $i$th component $w_i$ is the weight of the curve corresponding to the edge $i$ of $T_l$. By Theorem~\ref{thm:Xcanonicalregular}, we know that $1\otimes M_{T_l^\circ}(\mathbf{w})$ is a Laurent polynomial in the expressions $1\otimes M_{T_l^\circ}(\sum_i\varepsilon_{ki}\mathbf{e}_i)=X_k\mathbb{B}_k$. Indeed, it is the image of $\mathbb{I}_{\mathcal{A}}^q(l)$ under the map $\pi^q$. Let us write $\mathbf{w}_+=(w_1^+,\dots,w_m^+)$ for the vector whose $i$th component equals $w_i$ if $w_i\geq0$ and zero otherwise. Let us write $\mathbf{w}_-=(w_1^-,\dots,w_m^-)$ for the vector whose $i$th component equals~$w_i$ if $w_i\leq0$ and zero otherwise. Then in the algebra $\mathcal{F}_{\mathcal{D}}$ defined above, we have 
\begin{align*}
1\otimes M_{T_l^\circ}(\mathbf{w}) &= 1\otimes \omega^{\Lambda_{T_l^\circ}(\mathbf{w}_-,\mathbf{w}_+)}M_{T_l^\circ}(\mathbf{w}_-)M_{T_l^\circ}(\mathbf{w}_+) \\
&= 1\otimes\omega^{\Lambda_{T_l^\circ}(\mathbf{w}_-,\mathbf{w}_+)}\omega^{\sum_{i<j}\Lambda_{T_l^\circ}(\mathbf{e}_i,\mathbf{e}_j)w_i^-w_j^-}\prod_{i=1}^mM_{T_l^\circ}(\mathbf{e}_i)^{w_i^-}\cdot M_{T_l^\circ}(\mathbf{w}_+) \\
&= \omega^{\Lambda_{T_l^\circ}(\mathbf{w}_-,\mathbf{w}_+)}\omega^{-\sum_{i<j}\Lambda_{T_l}(\mathbf{e}_i,\mathbf{e}_j)w_i^-w_j^-}\prod_{i=1}^mM_{T_l}(\mathbf{e}_{m-i+1})^{w_{m-i+1}^-}\otimes M_{T_l^\circ}(\mathbf{w}_+) \\
&= \omega^{\Lambda_{T_l^\circ}(\mathbf{w}_-,\mathbf{w}_+)}M_{T_l}(-\mathbf{w}_-)^{-1}\otimes M_{T_l^\circ}(\mathbf{w}_+).
\end{align*}
In the third step of this calculation, we have factored out $\prod_{i=1}^m(M_{T_l}(-\mathbf{e}_i)\otimes M_{T_l^\circ}(\mathbf{e}_i))^{w_i^-}$, which equals~1 in $\mathcal{F}_{\mathcal{D}}$. It is easy to see that the expression in the last line of this calculation is $(\mathbb{I}_{\mathcal{D}}^q\circ\varphi)(l)$. Thus we have 
\[
(\mathbb{I}_{\mathcal{D}}^q\circ\varphi)(l)=(\pi^q\circ\mathbb{I}_{\mathcal{A}}^q)(l)
\]
as desired.
\end{proof}

\appendix
\section{Derivation of mutation formulas}
\label{app:DerivationOfMutationFormulas}

In this appendix, we calculate the action of the map $\mu_k^q$ of Definition~\ref{def:doublemutation} on the generators~$B_i$ and~$X_i$.

\begin{lemma}
\label{lem:commuteseries}
Let $\varphi(x)$ be any formal power series in $x$. Then 
\[
\varphi(X_k)B_i = B_i\varphi(q^{2\delta_{ik}}X_k), \quad
\varphi(\widehat{X}_k)B_i = B_i\varphi(q^{2\delta_{ik}}\widehat{X}_k),
\]
and
\[
\varphi(X_k)X_i = X_i\varphi(q^{2\varepsilon_{ki}}X_k).
\]
\end{lemma}

\begin{proof}
Write $\varphi(x)=a_0+a_1x+a_2x^2+\dots$. Then the relation $X_kX_i=q^{2\varepsilon_{ki}}X_iX_k$ implies 
\begin{align*}
\varphi(X_k)X_i &= a_0X_i+a_1X_kX_i+a_2X_k^2X_i+\dots \\
&= X_ia_0+X_ia_1q^{2\varepsilon_{ki}}X_k+X_ia_2q^{4\varepsilon_{ki}}X_k^2+\dots \\
&= X_i\left(a_0+a_1(q^{2\varepsilon_{ki}}X_k)+a_2(q^{2\varepsilon_{ki}}X_k)^2+\dots\right) \\
&= X_i\varphi(q^{2\varepsilon_{ki}}X_k).
\end{align*}
This proves the third relation. The first two relations are proved similarly. For these one uses the facts $X_kB_i=q^{2\delta_{ik}}B_iX_k$ and $\widehat{X}_kB_i=q^{2\delta_{ik}}B_i\widehat{X}_k$.
\end{proof}

\begin{proposition}
The map $\mu_k^\sharp$ is given on generators by the formulas 
\[
\mu_k^\sharp(B_i)=
\begin{cases}
B_k(1+qX_k)(1+q\widehat{X}_k)^{-1} & \mbox{if } i=k \\
B_i & \mbox{if } i\neq k
\end{cases}
\]
and 
\[
\mu_k^\sharp(X_i)=
\begin{cases}
X_i(1+qX_k)(1+q^3X_k)\dots(1+q^{2|\varepsilon_{ik}|-1}X_k) & \mbox{if } \varepsilon_{ik}\leq0 \\
X_i{\left((1+q^{-1}X_k)(1+q^{-3}X_k)\dots(1+q^{1-2|\varepsilon_{ik}|}X_k)\right)}^{-1} & \mbox{if } \varepsilon_{ik}\geq0.
\end{cases}
\]
\end{proposition}

\begin{proof}
By Lemma~\ref{lem:commuteseries}, we have 
\begin{align*}
\mu_k^\sharp(B_i) &= \Psi^q(X_k)\Psi^q(\widehat{X}_k)^{-1}B_i\Psi^q(\widehat{X}_k)\Psi^q(X_k)^{-1} \\
&= B_i\Psi^q(q^{2\delta_{ik}}X_k)\Psi^q(q^{2\delta_{ik}}\widehat{X}_k)^{-1}\Psi^q(\widehat{X}_k)\Psi^q(X_k)^{-1}. 
\end{align*}
If $i\neq k$, then this equals $B_i$ as desired. Suppose on the other hand that $i=k$. Using the identity $\Psi^q(q^2x)=(1+qx)\Psi^q(x)$ and commutativity of~$X_i$ and~$\widehat{X}_k$, we can rewrite this last expression as
\begin{align*}
\mu_k^\sharp(B_i) &= B_k\Psi^q(q^2X_k)\Psi^q(q^2\widehat{X}_k)^{-1}\Psi^q(\widehat{X}_k)\Psi^q(X_k)^{-1} \\
&= B_k(1+qX_k)\Psi^q(X_k) \left((1+q\widehat{X}_k)\Psi^q(\widehat{X}_k)\right)^{-1} \Psi^q(\widehat{X}_k)\Psi^q(X_k)^{-1} \\
&= B_k(1+qX_k)(1+q\widehat{X}_k)^{-1}.
\end{align*}
This completes the proof of the first formula.

To prove the second formula, observe that by Lemma~\ref{lem:commuteseries} and the commutativity of~$X_i$ and~$\widehat{X}_k$ we have 
\begin{align*}
\mu_k^\sharp(X_i) &= \Psi^q(X_k)\Psi^q(\widehat{X}_k)^{-1}X_i\Psi^q(\widehat X_k)\Psi^q(X_k)^{-1} \\
&= \Psi^q(X_k)X_i\Psi^q(X_k)^{-1} \\
&= X_i\Psi^q(q^{-2\varepsilon_{ik}}X_k)\Psi^q(X_k)^{-1}.
\end{align*}
If $\varepsilon_{ik}\leq0$, then the identity $\Psi^q(q^2x)=(1+qx)\Psi^q(x)$ implies 
\begin{align*}
\mu_k^\sharp(X_i) &= X_i\Psi^q(q^2q^{2|\varepsilon_{ik}|-2}X_k)\Psi^q(X_k)^{-1} \\
&= X_i(1+q^{2|\varepsilon_{ik}|-1}X_k)\Psi^q(q^{2|\varepsilon_{ik}|-2}X_k)\Psi^q(X_k)^{-1} \\
&= X_i(1+q^{2|\varepsilon_{ik}|-1}X_k)(1+q^{2|\varepsilon_{ik}|-3}X_k)\Psi^q(q^{2|\varepsilon_{ik}|-4}X_k)\Psi^q(X_k)^{-1} \\
&= \dots \\
&= X_i(1+q^{2|\varepsilon_{ik}|-1}X_k)\dots(1+q^3X_k)(1+qX_k)
\end{align*}
as desired. If $\varepsilon_{ik}\geq0$, there is a similar argument using the identity $\Psi^q(q^{-2}x)=(1+q^{-1}x)^{-1}\Psi^q(x)$.
\end{proof}

\begin{lemma}
\label{lem:dualmutation}
Let $(L,\{e_i\}_{i\in I},\{e_j\}_{j\in J},(\cdot,\cdot))$ be a seed. If we mutate this seed in the direction of a basis vector~$e_k$, then the basis $\{f_i\}$ for $L^\vee$ transforms to a new basis $\{f_i'\}$ given by 
\[
f_i'=
\begin{cases}
-f_i+\sum_j[-\varepsilon_{kj}]_+f_j & \mbox{if } i=k \\
f_i & \mbox{if } i\neq k.
\end{cases}
\]
\end{lemma}

\begin{proof}
The transformation $e_i\mapsto e_i'$ can be represented by an explicit matrix by Definition~\ref{def:altmutation}, and the transformation rule appearing in the lemma is represented by the transpose of this matrix.
\end{proof}

\begin{proposition}
The map $\mu_k'$ is given on generators by the formulas 
\[
\mu_k'(B_i')=
\begin{cases}
\mathbb{B}_k^-/B_k & \mbox{if } i=k \\
B_i & \mbox{if } i\neq k
\end{cases}
\]
and 
\[
\mu_k'(X_i')=
\begin{cases}
X_k^{-1} & \mbox{if } i=k \\
q^{-[\varepsilon_{ik}]_+\varepsilon_{ik}}X_iX_k^{[\varepsilon_{ik}]_+} & \mbox{if } i\neq k.
\end{cases}
\]
\end{proposition}

\begin{proof}
Let $Y_v$ be the generator of $\mathcal{D}_{\mathbf{i}}$ associated to $v\in L_{\mathcal{D}}$ as in Definition~\ref{def:quantumtorus}. By Lemma~\ref{lem:dualmutation} and the fact that $(f_i,f_j)_{\mathcal{D}}=0$ for all $i$,~$j$, we have 
\begin{align*}
\mu_k'(B_k') &= Y_{-f_k+\sum_j[-\varepsilon_{kj}]_+f_j} \\
&= B_k^{-1}\mathbb{B}_k^{-}.
\end{align*}
Similarly we have 
\begin{align*}
\mu_k'(X_i') &= Y_{e_i+[\varepsilon_{ik}]_+e_k} \\
&= q^{-[\varepsilon_{ik}]_+(e_i,e_k)_{\mathcal{D}}}Y_{e_i}Y_{[\varepsilon_{ik}]_+e_k} \\
&= q^{-[\varepsilon_{ik}]_+\varepsilon_{ik}}X_iX_k^{[\varepsilon_{ik}]_+}
\end{align*}
for $i\neq k$ and $\mu_k'(X_i')=Y_{-e_k}=X_k^{-1}$ for $i=k$.
\end{proof}

\begin{theorem}
The map $\mu_k^q$ is given on generators by the formulas 
\[
\mu_k^q(B_i')=
\begin{cases}
(qX_k\mathbb{B}_k^++\mathbb{B}_k^-)B_k^{-1}(1+q^{-1}X_k)^{-1} & \mbox{if } i=k \\
B_i & \mbox{if } i\neq k
\end{cases}
\]
and
\[
\mu_k^q(X_i')=
\begin{cases}
X_i\prod_{p=0}^{|\varepsilon_{ik}|-1}(1+q^{2p+1}X_k) & \mbox{if } \varepsilon_{ik}\leq0 \mbox{ and } i\neq k \\
X_iX_k^{\varepsilon_{ik}}\prod_{p=0}^{\varepsilon_{ik}-1}(X_k+q^{2p+1})^{-1} & \mbox{if } \varepsilon_{ik}\geq0 \mbox{ and } i\neq k \\
X_k^{-1} & \mbox{if } i=k.
\end{cases}
\]
\end{theorem}

\begin{proof}
By our formulas for $\mu_k'$ and $\mu_k^\sharp$, we have 
\begin{align*}
\mu_k^q(B_k') &= \mu_k^\sharp(\mu_k'(B_k')) = \mu_k^\sharp(\mathbb{B}_k^-/B_k) \\
&= \mathbb{B}_k^-\left(B_k(1+qX_k)(1+q\widehat{X}_k)^{-1}\right)^{-1} \\
&= \mathbb{B}_k^-(1+q\widehat{X}_k)(1+qX_k)^{-1}B_k^{-1}.
\end{align*}
By the definitions of $\mathbb{B}_k^-$ and $\widehat{X}_k$, this equals
\begin{align*}
\mu_k^q(B_k') &= \prod_{i|\varepsilon_{ki}<0}B_i^{-\varepsilon_{ki}}\big(1+qX_k\prod_iB_i^{\varepsilon_{ki}}\big)(1+qX_k)^{-1}B_k^{-1} \\
&= \big(\prod_{i|\varepsilon_{ki}<0}B_i^{-\varepsilon_{ki}}+qX_k\prod_{i|\varepsilon_{ki}>0}B_i^{\varepsilon_{ki}}\big)(1+qX_k)^{-1}B_k^{-1} \\
&=(qX_k\mathbb{B}_k^++\mathbb{B}_k^-)(1+qX_k)^{-1}B_k^{-1} \\
&=(qX_k\mathbb{B}_k^++\mathbb{B}_k^-)B_k^{-1}(1+q^{-1}X_k)^{-1}.
\end{align*}
From this calculation, we easily obtain the formula describing the action of the map $\mu_k^q$ on the generators $B_i'$.

On the other hand, if $\varepsilon_{ik}\leq0$ and $i\neq k$, then 
\begin{align*}
\mu_k^q(X_i') &= \mu_k^\sharp(\mu_k'(X_i')) = \mu_k^\sharp(X_i) \\
&= X_i(1+qX_k)(1+q^3X_k)\dots(1+q^{2|\varepsilon_{ik}|-1}X_k) \\
&= X_i\prod_{p=0}^{|\varepsilon_{ik}|-1}(1+q^{2p+1}X_k).
\end{align*}
If $\varepsilon_{ik}\geq0$ and $i\neq k$, then 
\begin{align*}
\mu_k^q(X_i') &= \mu_k^\sharp(\mu_k'(X_i')) = \mu_k^\sharp(q^{-\varepsilon_{ik}^2}X_iX_k^{\varepsilon_{ik}}) \\
&= q^{-\varepsilon_{ik}^2}X_i{\left((1+q^{-1}X_k)(1+q^{-3}X_k)\dots(1+q^{1-2|\varepsilon_{ik}|}X_k)\right)}^{-1}X_k^{\varepsilon_{ik}} \\
&= q^{-\varepsilon_{ik}^2}X_i{\left(q^{-1}(X_k+q)q^{-3}(X_k+q^3)\dots q^{1-2|\varepsilon_{ik}|}(X_k+q^{2|\varepsilon_{ik}|-1})\right)}^{-1}X_k^{\varepsilon_{ik}}.
\end{align*}
Applying the identity $1+3+5+\dots+(2N-1)=N^2$, this becomes 
\begin{align*}
\mu_k^q(X_i') &= X_i{\left((X_k+q)(X_k+q^3)\dots (X_k+q^{2|\varepsilon_{ik}|-1})\right)}^{-1}X_k^{\varepsilon_{ik}} \\
&= X_iX_k^{\varepsilon_{ik}} (X_k+q)^{-1}(X_k+q^3)^{-1}\dots (X_k+q^{2|\varepsilon_{ik}|-1})^{-1} \\
&= X_iX_k^{\varepsilon_{ik}}\prod_{p=0}^{\varepsilon_{ik}-1}(X_k+q^{2p+1})^{-1}.
\end{align*}
Finally, if $i=k$, then we have 
\[
\mu_k^q(X_i') = \mu_k^\sharp(\mu_k'(X_i')) = \mu_k^\sharp(X_k^{-1}) = X_k^{-1}.
\]
This proves the second formula.
\end{proof}

\section*{Acknowledgments}
\addcontentsline{toc}{section}{Acknowledgements}

I thank Ben~Davison, Alexander~Goncharov, Hyun~Kyu~Kim, Greg~Muller, and Linhui~Shen for helpful discussions.


\begin{thebibliography}{99}
\addcontentsline{toc}{section}{References}

\bibitem{AK} Allegretti, D.G.L. and Kim, H.K. (2017). A duality map for quantum cluster varieties from surfaces. \emph{Advances in Mathematics}, \textbf{306}, 1164--1208.

\bibitem{Dmoduli} Allegretti, D.G.L. (2017). The cluster symplectic double and moduli spaces of local systems. \emph{Proceedings of the AMS}, \textbf{145}(2017), 5191--5204.

\bibitem{Dlam} Allegretti, D.G.L. (2019). Laminations from the symplectic double. \emph{Geometriae Dedicata}, \textbf{199}(1), 27--86.

\bibitem{categorification} Allegretti, D.G.L. (2019). Categorified canonical bases and framed BPS states. \emph{Selecta Mathematica}, \textbf{25}(69).

\bibitem{BZq} Berenstein, A. and Zelevinsky, A. (2005). Quantum cluster algebras. \emph{Advances in Mathematics}, \textbf{195}(2), 405--455.

\bibitem{BonahonWong} Bonahon, F. and Wong, H. (2011). Quantum traces for representations of surface groups in $SL_2(\mathbb{C})$. \emph{Geometry and Topology}, \textbf{15}(3), 1569--1615.

\bibitem{IHES} Fock, V.V. and Goncharov, A.B. (2006). Moduli spaces of local systems and higher Teichm\"uller theory. \emph{Publications Math\'ematiques de l'Institut des Hautes \'Etudes Scientifiques}, \textbf{103}(1), 1--211.

\bibitem{dual} Fock, V.V. and Goncharov, A.B. (2007). Dual Teichm\"uller and lamination spaces. In Handbook of Teichm\"uller theory I, \emph{IRMA Lectures in Mathematics and Theoretical Physics}, \textbf{11}, 647--684.

\bibitem{ensembles} Fock, V.V. and Goncharov, A.B. (2009). Cluster ensembles, quantization and the dilogarithm. \emph{Annales Scientifiques de l'\'Ecole Normale Sup\'erieure}, \textbf{42}(6), 865--930.

\bibitem{dilog} Fock, V.V. and Goncharov, A.B. (2009). The quantum dilogarithm and representations of quantum cluster varieties. \emph{Inventiones mathematicae}, \textbf{175}(2), 223--286.

\bibitem{double} Fock, V.V. and Goncharov, A.B. (2016). Symplectic double for moduli spaces of $G$-local systems on surfaces. \emph{Advances in Mathematics}, \textbf{300}, 505--543.

\bibitem{FZI} Fomin, S. and Zelevinsky, A. (2002). Cluster algebras I: Foundations. \emph{Journal of the American Mathematical Society}, \textbf{15}(2), 497--529.

\bibitem{FZIV} Fomin, S. and Zelevinsky, A. (2007). Cluster algebras IV: Coefficients. \emph{Compositio Mathematica}, \textbf{143}(01), 112--164.

\bibitem{GMN} Gaiotto, D., Moore, G.W., and Neitzke, A. (2013). Framed BPS states. \emph{Advances in Theoretical and Mathematical Physics}, \textbf{17}(2), 241--397.

\bibitem{GHKK} Gross, M., Hacking, P., Keel, S., and Kontsevich, M. (2018). Canonical bases for cluster algebras. \emph{Journal of the American Mathematical Society}, \textbf{31}(2), 497--608.

\bibitem{Le} L\^{e}, T. (2017). Quantum Teichm\"uller spaces and quantum trace map. \emph{Journal of the Institute of Mathematics of Jussieu}, 1--43.

\bibitem{Lusztig} Lusztig, G. (1990). Canonical bases arising from quantized enveloping algebras. \emph{Journal of the American Mathematical Society}, \textbf{3}(2), 447--498.

\bibitem{Muller} Muller, G. (2016). Skein algebras and cluster algebras of marked surfaces. \emph{Quantum Topology}, \textbf{7}(3), 435--503

\bibitem{MSW1} Musiker, G. Schiffler, R., and Williams, L. (2011). Positivity for cluster algebras from surfaces. \emph{Advances in Mathematics}, \textbf{227}(6), 2241--2308.

\bibitem{MSW2} Musiker, G. Schiffler, R., and Williams, L. (2013). Bases for cluster algebras from surfaces. \emph{Compositio Mathematica}, \textbf{149}(2), 217--263.

\bibitem{Thurston} Thurston, D. (2014). Positive basis for surface skein algebras. \emph{Proceedings of the National Academy of Sciences}, \textbf{111}(27), 9725--9732.

\bibitem{Tran} Tran, T. (2011). $F$-polynomials in quantum cluster algebras. \emph{Algebras and representation theory}, \textbf{14}(6), 1025--1061.

\end{thebibliography}
\end{document}